\documentclass{article}

\usepackage{amsmath,amsfonts,calrsfs,fullpage,amssymb,color,verbatim,eucal,yfonts,mathrsfs,titling}
\usepackage{nicefrac}
\usepackage{bbm}
\usepackage{authblk}

\def\R{{\mathbb R}}
\def\N{{\mathbb N}}

\def\B{{\mathscr B}}

\def\1{\mathbbm{1}}
\def\p{\textfrak{p}}

\newtheorem{theorem}{Theorem}[section]
\newtheorem{proposition}[theorem]{Proposition}

\newtheorem{lemma}[theorem]{Lemma}
\newtheorem{rmk}[theorem]{Remark}

\newtheorem{definition}[theorem]{Definition}

\newenvironment{proof}{{\sc Proof.}}{\hfill {\sc qed}}
\newenvironment{remark}{\begin{rmk}\rm}{\end{rmk}}
\newcommand\scal[2]{\langle #1,#2\rangle}

\newcommand{\res}{\mathop{\hbox{\vrule height 7pt width .5pt depth 0pt
\vrule height .5pt width 6pt depth 0pt}}\nolimits}

\title{On integration by parts formula on open convex sets in Wiener spaces}
\author{Davide Addona\thanks{email: davide.addona@unimib.it}}
\affil{Department of Mathematics and applications\\
University of Milano Bicocca\\
via Cozzi 55, 20125 Milano, Italy}
\author{Giorgio Menegatti\thanks{email:
giorgio.menegatti@unife.it} }  
\author{Michele Miranda jr.\thanks{email: michele.miranda@unife.it}}
\affil{Department of Mathematics and Computer Sciences\\
University of Ferrara\\
via Machiavelli 30, I-44121 Ferrara, Italy}

\providecommand{\keywords}[1]{{\textit{Keywords}:} #1}
\providecommand{\subjclass}[1]{{\textit{SubjClass}[2000]:} #1}
\begin{document}

\maketitle
\abstract{In Euclidean space, it is well known that any integration by parts formula for a set of finite perimeter $\Omega$
is expressed by the integration with respect to a measure $P(\Omega,\cdot)$ which is equivalent to the one-codimensional Hausdorff measure 
restricted to the reduced boundary of $\Omega$. The same result has been proved in an abstract Wiener space, typically an infinite dimensional space, where the surface measure considered is the one-codimensional spherical Hausdorff-Gauss measure $\mathscr S^{\infty-1}$ restricted to the measure-theoretic boundary of $\Omega$. In this paper we consider an open convex set $\Omega$ and we provide an explicit formula for the density of $P(\Omega,\cdot)$ with respect to $\mathscr S^{\infty-1}$. In particular, the density can be written in terms of the Minkowski functional $\p$ of $\Omega$ with respect to an inner point of $\Omega$. As a consequence, we obtain an integration by parts formula for open convex sets in Wiener spaces.}

\vspace{0.5cm}
\\


\section{Introduction}
We consider a separable Banach space $X$ endowed with a Gaussian measure $\gamma$, whose Cameron-Martin space is denoted by $H$. The covariance operator is denoted by $Q:X^*\rightarrow X$, where $X^*$ is the topological dual of $X$, and $\Omega\subseteq X$ is an open and convex domain. The aim of this paper is proving an integration-by-parts formula for the domain $\Omega$. To be more precise, we are going to show that for any Lipschitz function $\psi:X\rightarrow\R$ it holds that
\begin{align}
\label{ansimare}
\int_\Omega \partial_k^* \psi d\gamma=\int_{\partial \Omega}\psi \frac{\partial_k\p}{|\nabla_H\p|_H}d\mathscr S^{\infty-1}, \qquad k\in\N.
\end{align}
Here, $\p$ is the Minkowski functional of $\Omega$ and $\mathscr S^{\infty-1}$ is the (spherical) Hausdorff-Gauss surface measure introduced in \cite{FD92}, where the surface measure is denoted by $\rho$. However, we use the notation $\mathscr S^{\infty-1}$ which has been introduced in \cite{AMP10} and is more familiar with the language of geometric measure theory. The measure $\rho$ is the generalization of the Airault-Malliavin surface measure \cite{AM88}.

The crucial tools to obtain formula \eqref{ansimare} are convex analysis and geometric measure theory in infinite dimension. 
The former ensures that the Minkowski functional $\p$ related to the open convex domain $\Omega$ satisfies regularity conditions. 
Indeed, it is well known that the Minkowski functional related to an open convex set is convex and continuous (see \cite{Phe93}) and therefore 
we infer that $\p$ is Lipschitz, and therefore  G\^ateaux differentiable almost everywhere. This allows us to write the exterior normal vector of $\Omega$ in terms of $\p$, 
as in finite dimensional setting.

Geometric measure theory has been recently developed, starting from the first definition of functions of bounded variation ($BV$ functions for short)
in abstract Wiener spaces (which we denote by $BV(X,\gamma)$) given by \cite{F00} and \cite{FH01}. However, the authors propose a stochastic approach, 
defining the sets of finite perimeter in terms of reflected Brownian motions and by using the theory of Dirichlet forms. 
In \cite{AmbMirManPal10} the authors prove the results of \cite{FH01} and further properties of $BV$ functions in abstract Wiener spaces in 
a purely analytic setting, closer to the classical one. In particular, they prove the equivalence between different definitions of $BV(X,\gamma)$ 
in terms of total variation $V_H(f)$ of a function $f$, by approximation with more regular functions throughout the functional $L_H(f)$ and 
by means of the Ornstein-Uhlenbeck semigroup $(T_t)_{t\geq0}$. The latter is the analogous in the Gaussian setting of the heat semigroup in 
the original definition of $BV$ functions given by De Giorgi in \cite{DG53}.

We recall the definition of the space $BV(X,\gamma)$ of the functions of bounded variation on $X$ (see e.g. \cite{FH01} and \cite[Definition 3.1]{AmbMirManPal10}). We say that $f\in L^1(\log L)^{1/2}(X,\gamma)$ is a function of bounded variation if there exists a finite signed Radon measure $\mu\in\mathscr M(X;H)$ such that for any $h\in QX^*$ it follows that
\begin{align*}
\int_Xf\partial_h^*\Psi d\gamma=-\int_X\Psi d[h,\mu]_H,
\end{align*}
for any $\Psi\in \mathcal{FC}_b^1(X)$. Further, if $U\subset X$ is a Borel set and $f=\1_U$, if $f$ has bounded variation then we say that $U$ has finite perimeter and we denote by $P(U,\cdot)$ the associated measure. The definition of $BV$ functions on an open set $A\subset X$ is more complicated, since of the lackness of local compactness in infinite dimension. However, $BV$ functions on open domains $A$ has been investigated in \cite{AdMeMi18}, where, as in \cite{AmbFusPal00}, the authors provide different characterizations of the space $BV(A,\gamma)$ by means of the total variation $V_\gamma(f,A)$ and in terms of approximations with more regular functions throughout the functional $L_\gamma(f,A)$. We stress that the characterization in terms of the Ornstein-Uhlenbeck semigroup of $BV(A,\gamma)$ is not an easy task since at the best of our knowledge there is no good definition of $(T_t)_{t\geq0}$ on a general open domain $A$. However, in \cite{Cappa} it has been defined the Ornstein-Uhlenbeck semigroup $(T^C_t)_{t\geq0}$ on the convex set $C\subset X$ by means of finite dimens$^{•}$ional approximations, and in \cite{LMP15} the authors relate the variation of a function $f$ with the behaviour of $T_t^Cu$ near $0$. 

Sets of finite perimeter play a crucial role in our investigation. As in the finite dimensional case, the measure associated to sets 
of finite perimeter is strictly connected with a surface measure.  In \cite{FD92} it is introduced a notion of surface measure in 
infinite dimension, the spherical Hausdorff-Gauss surface measure $\mathscr S^{\infty-1}$, which is defined by means of finite dimensional 
spherical Hausdorff measure 
$\mathscr S^{n-1}$, $n\in\N$. This is different from the classical Hausdorff measure $\mathscr H^{n-1}$ even if the relation 
$\mathscr H^{n-1}\leq \mathscr S^{n-1}\leq 2\mathscr H^{n-1}$ holds true and they coincide on rectifiable sets. This choice is due to the fact that 
spherical Hausdorff-Gauss surface measure $\mathscr S^{n-1}$ enjoy a monotonicity property (see  \cite[Lemma 3.2]{AMP10}, \cite[Proposition 6(ii)]{FD92} 
or \cite[Proposition 2.4]{HI10}) which allows to define measure $\mathscr S^{\infty-1}$ as limit on direct sets. Further details are given 
in Section \ref{crumiro}.

Properties of sets of finite perimeter have been widely studied in \cite{AMP10}, \cite{CLMN12} and \cite{HI10}. 
In particular, \cite[Theorem 5.2]{AMP10} and \cite[Theorem 2.11]{HI10} show that if $U$ has finite perimeter in $X$, 
then $P(U,B)=\mathscr S^{\infty-1}(B\cap\partial^* U)$, where $\partial ^*U$ is the cylindrical essential boundary introduced 
in \cite[Definition 2.9]{HI10}. 
It is worth noticing that in the infinite-dimensional setting things do not work as
well as for the Euclidean case; \cite{P81} gives an example of an infinite-dimensional
Hilbert space $X$, a Gaussian measure $\gamma$ and a set $E \subset X$ such that $0 < \gamma(E) < 1$ and
\begin{align}
\label{copiato}
\lim_{r\rightarrow0}\frac{\gamma(E\cap B_r(x))}{\gamma(B_r(x))}=1, \mbox{ for every } x\in X.
\end{align}
In the same work, it is also shown that if the eigenvalues of the covariance $Q$ decay to
zero sufficiently fast, then it is possible to talk about density points; in some sense,
the requirement on the decay gives properties of $X$ closer to the finite-dimensional
case. For these reasons, in general the notion of point of density as given in \eqref{copiato} is
not a good notion.
However, \cite{AF11} gives a definition of points of density $1/2$ by means of the Ornstein-Uhlenbeck semigroup $(T_t)_{t\geq0}$.

The properties of $\Omega$ give other important consequences. At first, we show that, as in finite dimension, for any open convex set $C\subset X$ we have $\partial C=\partial ^*C$, where $\partial C$ denotes the topological boundary of $C$. Further, from \cite[Proposition 9]{CLMN12}, it follows that $\Omega$ has finite perimeter and therefore from the above reasoning it follows that $P(\Omega,B)=\mathscr S^{\infty-1}(B\cap \partial^*\Omega)=\mathscr S^{\infty-1}(B\cap\partial \Omega)$. This explain why in the right-hand side of \eqref{ansimare} the measure $\mathscr S^{\infty-1}\res\partial\Omega$ appears.

Finally, we stress that \eqref{ansimare} is the generalization of the integration-by-parts formula proved in \cite{CL14}. Here, the authors deal with subsets of $X$ of the type $\mathcal O:=\{x\in X:G(x)<0\}$, where $G:X\rightarrow \R$ is a suitable regular function which satisfy a sort of nondegeneracy assumption, and they prove that
\begin{align}
\label{sospiro}
\int_{\mathcal O}\partial_k^*\varphi d\gamma=\int_{G^{-1}(0)}\varphi \frac{\partial_kG}{|\nabla_HG|_H}\varphi d\mathscr S^{\infty-1}, \qquad k\in\N,
\end{align} 
for any Lipschitz function $\varphi:X\rightarrow\R$.  $G^{-1}(0)$ coincides $\mathscr S^{\infty-1}$-almost everywhere with 
$\partial\mathcal O$. Thanks to \eqref{sospiro}, 
the authors set the bases of a theory of the traces for Sobolev functions in abstract Wiener spaces and proved the existence of a trace operator 
${\rm Tr}$. However, this theory if far away to be complete. Indeed, in general if $f$ belongs to the Sobolev space $W^{1,p}(\mathcal O,\gamma)$ 
with $p\in(1,+\infty)$ (see \cite{CL14} for the definition of $W^{1,p}(\mathcal O,\gamma)$), then ${\rm Tr}f\in L^q(\partial \mathcal O,\rho)$ with $1\leq q<p$. 
The case $q=p$ is still an open problem, and in this direction some result is known if $G$ satisfies some additional conditions, 
which are not even fulfilled by the balls in Hilbert spaces. We recall that in the case $\mathcal O=X$ the surface integral in \eqref{sospiro} 
disappears and therefore \eqref{sospiro} is the usual integration-by-parts formula in abstract Wiener space (see e.g. \cite[Chapter 5]{Bog98}).

Comparing \eqref{ansimare} and \eqref{sospiro} we notice that the Minkowski functional $\p$ of $\Omega$ plays the role of the function $G$ in \cite{CL14}. However, $\p$ in general does not satisfies the assumptions of \cite{CL14} for $G$ and in this sense our result is a generalization of \eqref{sospiro}.
Moreover, our work suggests a different way to get the integration-by-parts formula by using procedures and techniques inherit from the geometric measure theory. 
This different approach gives the hope to develop in future papers a more general trace theory for Sobolev and BV functions in abstract Wiener spaces.

The paper is organized as follows.

In Section \ref{preliminaries} we define the abstract Wiener space $(X,\gamma,H)$ and the main tools of differential calculus in infinite dimension, i.e., the $H$-gradient, the $\gamma$- divergence and the Sobolev spaces $W^{1,p}(\Omega,\gamma)$, with $p\in[1,+\infty)$. Moreover, we recall the definition of functions of bounded variation both on $X$ and on an open set $A\subset X$.

In Section \ref{crumiro} we recall the definition of $\mathscr S^{\infty-1}$ and, thanks to an infinite dimensional version of the area formula, 
we prove that the epigraph of a Sobolev function has finite perimeter.

Finally, in Section \ref{coppadelmondo} we prove the integration-by-parts formula \eqref{ansimare}. To this aim we initially show that, thanks
to \cite[Lemma 6.3]{AMP10}, it is possible to choice a direction $h\in QX^*$ such that
$|D_\gamma\1_\Omega|\left(\left\{x\in X:[\nu_\Omega(x),h]_H=0\right\}\right)=0$, where $\nu_\Omega$ is the Radon-Nikodym density of
$D_\gamma\1_\Omega$ with respect $|D_\gamma\1_\Omega|$, i.e., $D_\gamma\1_\Omega=\nu_\Omega|D_\gamma\1_\Omega|$. 
We set $\Omega_h^\perp:=\{x\in \Omega:\hat h(x)=0\}$, where $\hat h\in X^*$ satisfies $h=Q\hat h$.
Then, there exist two functions $f,g:\Omega_h^\perp\rightarrow \R$ such that
$\partial \Omega=\Gamma(f,\Omega_h^\perp)\cup\Gamma(g,\Omega_h^\perp)\cup N$, where $N$ is a Borel set with null 
$\mathscr S^{\infty-1}$-measure and $\Gamma(f,\Omega_h^\perp):=\{y+f(y)h:y\in\Omega_h^\perp\}$. By applying the results of Section 
\ref{crumiro} it follows that $D_\gamma\1_\Omega=-\nu_f\mathscr S^{\infty-1}\res\Gamma(f,\Omega_h^\perp)+
\nu_g\mathscr S^{\infty-1}\res\Gamma(g,\Omega_h^\perp)$. To conclude, we show a relation between $\p$ and $f$ and $g$, which gives \eqref{ansimare}.

\section{Preliminaries}
\label{preliminaries}

Let us fix some notations. We denote by $(X,\gamma,H)$
an abstract Wiener space, i.e. a separable infinite
dimensional Banach space $X$ endowed with
a Radon centered non degenerate Gaussian measure $\gamma$ 
with Cameron--Martin
space $H$. We recall that $H$ is continuously and
compactly embedded in $X$ and that there
exists $Q:X^*\to X$ such that $QX^*\subset H\subset X$,
all these embeddings being dense by the non-degeneracy of $\gamma$. The decomposition
$Q=R_\gamma\circ j$ holds, where
$j:X^*\to L^2(X,\gamma)$ is just 
the identification of an element 
of $X^*$ as a function in $L^2(X,\gamma)$
and $R_\gamma:L^2(X,\gamma)\to X$ is defined
in terms of Bochner integral as
\[
R_\gamma(f)= \int_X f(x) x\, \gamma(dx).
\]
The reproducing kernel is defined as
\[
\mathscr{H}= \overline{j(X^*)}\subset L^2(X,\gamma),
\]
and the restriction of $R_\gamma$ on $\mathscr{H}$
gives a one--to--one correspondence between
$H$ and $\mathscr{H}$. For any $h\in H$ we shall
denote by $\hat h\in \mathscr{H}$ the unique
element such that $R_\gamma(\hat h)=h$.
Then, the Cameron--Martin space inherits 
the Hilbert structure with inner product
\[
[{h},{k}]_H=\int_X \hat h(x)\hat k(x)\gamma(dx).
\]

We denote by $\mathcal{F}C^1_b(X)$ the set 
of bounded functions $\varphi:X\to \R$ such that
there exists $n\in\N$, $x^*_1,\ldots x^*_n\in X^*$ 
and $v\in C^1_b(\R^n)$ (the space of bounded continuous functions with bounded continuous derivatives) with
\[
\varphi(x)=v(x^*_1(x),\ldots, x^*_n(x)).
\]
Without loss of generality, we can suppose that $Qx_1^*,\ldots,Qx_n^*$ are orthonormal vectors in $H$.
Further, we denote the set cylindrical $H$-valued vector fields by $\mathcal F C_b^1(X,H)$, where $\Phi\in\mathcal FC_b^1(X,H)$ if there exist $n\in\N$ and $h_1,\ldots, h_n\in H$ and $\varphi_1,\ldots,\varphi_n\in \mathcal F C_b^1(X)$ such that 
\begin{align*}
\Phi(x):=\sum_{i=1}^n\varphi_i(x) h_i.
\end{align*}

For any $\varphi\in \mathcal{F}C^1_b(X)$ and
$h\in H$ we set
\[
\partial_h \varphi(x)=\lim_{t\to 0}
\frac{\varphi(x+th)-\varphi(x)}{t}.
\]
{The separability of $X$ implies that $H$ is
separable. 

For any $\varphi\in\mathcal F C_b^1(X)$, $\varphi(x)=v(x_1^*(x),\ldots,x_n^*(x))$ for some $n\in\N$, $x_1,\ldots,x_n\in X^*$ and $v\in C_b^1(\R^n)$, we define its
$H$--gradient by 
\[
\nabla_H \varphi(x)=\sum_{i=1}^n
\partial_{Qx^*_i} \varphi(x) Qx_i^*,
\]
If $H'\subset H$ is a closed subspace and $Qx_i^*\in H'$ for any $i=1,\ldots,n$, then we write $\nabla_{H'}\varphi(x)=\nabla_H\varphi$ to enlight the dependence of $\varphi$ on the directions of $H'$.
For any $h\in H$ we also denote by
\[
\partial^*_h \varphi(x)=\partial_h \varphi(x)
-\varphi(x)\hat h(x)
\]
the formal adjoint (up to the sign) of 
$\partial_h$, in the sense that, 
for any $\varphi,\psi\in \mathcal{F}C^1_b(X)$, it holds that
\[
\int_X \varphi\partial_h \psi d\gamma
=-\int_X \partial^*_h \varphi \psi d\gamma.
\]

We introduce the divergence operator ${\rm div}_\gamma:\mathcal FC_b^1(X,H)\longrightarrow \R$ by setting
\begin{align*}
{\rm div}_\gamma \Phi(x):=\sum_{i=1}^n \partial_{h_i}\varphi_i(x)-\varphi_i(x)h_i,
\end{align*}
 with $\Phi(x)=\sum_{i=1}^n\varphi_i(x) h_i\in \mathcal FC_b^1(X,H)$. Further, for any $\Phi\in \mathcal FC_b^1(X,H)$ and any $\psi\in\mathcal FC_b^1(X)$ the following integration-by-parts formula holds:
 \begin{align*}
 \int_X[\nabla_H\psi,\Phi]_Hd\gamma=-\int_X\psi{\rm div}_\gamma \Phi d\gamma.
 \end{align*}

We stress that it is possible to fix an orthonormal
basis $\{h_i\}_{i\in \N}$ of $H$ such that $h_i=Qx^*_i$ with 
$x^*_i\in X^*$ for any $i\in \N$.}

For any $h\in QX^*$ we introduce the continuous projection $\pi_h:X\longrightarrow H$ defined by $\pi_hx=\hat h(x) h$, where $R_\gamma(\hat h)=h$. This fact induces the decomposition $X=X_h^\perp \oplus \langle h\rangle$, 
where $X_h^\perp ={\rm ker}({\pi_h})={\rm ker}(\hat h)$. Therefore, for any $x\in X$ we shall write $x=y+z$, where $y=x-\pi_hx\in X_h^\perp$ and $z=\pi_hx$. Clearly, this decomposition is unique.
Such a decomposition implies that the measure
$\gamma$ can be split as a product measure
\[
\gamma=\gamma_h^\perp \otimes \gamma_h
\]
where $\gamma_h^\perp$ and $\gamma_h$ are centred non-degenerate
Gaussian measures on $X_h^\perp$ and $\langle h\rangle$,
respectively. If $|h|_H=1$, then $\gamma_h$ is a standard Gaussian measure, i.e.
letting $z=th$, we have
\[
\gamma_h(dz)=\gamma_1(dt)=
\mathscr{N}(0,1)(dt)=\frac{1}{\sqrt{2\pi}}
e^{-\frac{t^2}{2}} dt.
\]

This argument can be generalized for any finite dimensional
subspace $F\subset QX^*\subset H$: indeed, if 
$F=\langle h_1,\ldots,h_m\rangle$ with
$\{h_i\}_{i=1,\ldots,m}$ orthonormal elements of $H$ and $h_i\in Q X^*$ for any $i=1,\ldots,m$, then we can write
$X=X_F^\perp \oplus F$, where $X_F^\perp ={\rm ker}(\pi_F)$,
$\pi_F:X\longrightarrow F$ and
\[
\pi_F(x)=\sum_{i=1}^m\hat h_i(x)h_i,
\]
and $\pi_F(h):=\sum_{i=1}^m[h,h_i]_Hh_i$ for any $h\in H$. We identify $F$ with $\R^m$ and for any $z\in F$ we denote by $|z|$ its norm in $\R^m$. We can also decompose $\gamma=\gamma_F^\perp\otimes
\gamma_F$ where $\gamma_F^\perp$ and $\gamma_F$ are 
centred non degenerate Gaussian measures on $X_F^\perp$ and $F$,
respectively. Further,{
\[
\gamma_F(dz) = 
\frac{1}{(2\pi)^{m/2}} e^{-\frac{|z|^2}{2}}\  dz.
\]
}
We recall the definition of Sobolev
spaces and functions with bounded variation
in Wiener spaces. Let $\Omega\subset
X$ be an open set. By ${\rm Lip}_b(\Omega)$ we denote the set of bounded Lipschitz continuous functions on $\Omega$, by ${\rm Lip}_c(\Omega)$ we
denote the set of functions 
$\varphi\in {\rm Lip}(X)$ {with bounded support and} such that ${\rm dist}({\rm supp}(\varphi), \Omega^c)>0$, and by {$\mathcal{F}C^1_b(\Omega)$ we
denote the set of restrictions of functions of $\mathcal{F}C^1_b(X)$ to $\Omega$}.
Clearly, ${\rm Lip}_c(\Omega)\subset{\rm Lip}_b(\Omega)$ and $\mathcal{F}C^1_b(\Omega)\subset{\rm Lip}_b(\Omega)$.
{Analogously, we define  ${\rm Lip}(\Omega, H)$ as the set of functions 
$\varphi:\ \Omega\rightarrow H$ such that there exists a positive constant $L$ which satisfies $|\varphi(x)-\varphi(y)|_H\leq L\|x-y\|_X$ for any $x,y\in X$.} ${\rm Lip}_b(\Omega,H)$ and ${\rm Lip}_c(\Omega,H)$ are defined in obvious way.

{We shall denote by $\mathscr{M}(\Omega,H)$ the
set of {$H$-valued} Borel measures defined on $\Omega\subset X$. For such measures the total
variation turns out to be given by
\begin{equation}
|\mu|(\Omega)=
\sup \left\{
\int_\Omega \scal{\Phi}{d\mu}_H: 
\Phi\in {\rm Lip}_{c}(\Omega,H), |\Phi(x)|_H\leq 1\ 
\forall x\in \Omega
\right\}. \label{norma}
\end{equation}
Equation \eqref{norma} has been proved in \cite[Lemma 2.3]{LMP15} with ${\rm Lip}_0(\Omega,H)$ instead of ${\rm Lip}_c(\Omega,H)$, but the same arguments can be adapted to prove \eqref{norma}.} We can state the following preliminary result.

\begin{lemma}
Let $1\leq p<\infty$. Then, the operator
$\nabla_H: \mathcal{F}C^1_b(\Omega)\subset L^p(\Omega,\gamma)
\to L^p(\Omega,\gamma,H)$ is closable. We denote by $W^{1,p}(\Omega,\gamma)$ the domain of its closure. {The same is true if we use ${\rm Lip}_{b}(\Omega)\subset L^p(\Omega,\gamma)$ instead of $\mathcal{F}C^1_b(\Omega)$, 
and the definition of $W^{1,p}(\Omega,\gamma)$ is equivalent.}
\end{lemma}

\begin{proof}
The above statement is true for $\Omega=X$ from \cite[Chapter 5]{Bog98}. By linearity, it is sufficient to prove
that if $f_j\to 0$ in $L^p(\Omega,\gamma)$ and $\nabla_H f_j
\to F$ in $L^p(\Omega,\gamma,H)$, then $F=0$. 
To this aim, we fix $\varphi\in {\rm Lip}_{c}(\Omega)$: notice that $|\nabla_H \varphi|_H\in L^\infty(X,\gamma)$, then the zero extension $\tilde\varphi=\varphi \cdot \1_\Omega$ of $\varphi$ belongs
to ${\rm Lip}(X)$ and $\partial^*_h \varphi\in L^\infty(X,\gamma)$ for any $h\in H$. Since $f_j\in \mathcal{F}
C^1_b(X)$ for any $j\in \N$ and $f_j\to 0$ in 
$L^p(\Omega,\gamma)$. Then we get
\begin{align*}
0=&
\lim_{j\to +\infty}
\int_\Omega f_j \partial^*_h \varphi d\gamma
=
\lim_{j\to +\infty}
\int_X f_j \partial^*_h \tilde\varphi d\gamma
=
\lim_{j\to +\infty}
-\int_X \partial_h f_j  \tilde \varphi d\gamma
=
\lim_{j\to +\infty}
-\int_\Omega \partial_h f_j  \varphi d\gamma \\
=& -\lim_{j\rightarrow+\infty}\int_\Omega[\nabla_hf_j,\varphi]_Hd\gamma
= -\int_\Omega [{F},{h}]_H  \varphi d\gamma,
\end{align*}
for any $h\in H$. Now, $\scal{F}{h}_H \in L^p(\Omega,\gamma)\subseteq L^1 (\Omega,\gamma)$, so
we can define $\mu\in \mathscr{M}(\Omega,H)$ by $\mu=\scal{F}{h}_H \gamma$. Therefore, \eqref{norma} gives $\mu\equiv0$. This implies
that $\scal{F}{h}_H=0$ $\gamma$--a.e. 
for every $h\in H$, and then $F=0$.

To prove the second part of the statement, we recall that the restriction to $\Omega$ of a function in $W^{1,p}(X,\gamma)$ is in $W^{1,p}(\Omega,\gamma)$, and by \cite[Chapter 5]{Bog98} we have ${\rm Lip}_b(X)\subseteq W^{1,p}(X,\gamma)$. Finally, a function in ${\rm Lip}_b(\Omega)$ can be extended to a function in ${\rm Lip}_b(X)$, and therefore  ${\rm Lip}_b(\Omega)\subseteq W^{1,p}(\Omega,\gamma)$ and we can conclude.
\end{proof}

From the definition of $W^{1,p}(\Omega,\gamma)$, it is easy to prove that
for any $f\in W^{1,p}(\Omega,\gamma)$ and $\Phi\in 
{\rm Lip}_{c}(\Omega,H)$ the following
integration by parts formula holds:
\[
\int_\Omega f{\rm div}_\gamma \Phi d\gamma
=-\int_\Omega \scal{\nabla_H f}{\Phi}_H d\gamma.
\]

We close this section by giving the definition
of functions of bounded variation both on $X$ and on open domains. For precise study of such functions see \cite{AdMeMi18}. We recall the definition on $X$.

\begin{definition}\label{jane}
{Let} $p>1$. We say that $u\in L^{p}(X,\gamma)$ is a function with bounded 
variation, i.e., $u\in BV(X,\gamma)$, if there
exists a Borel measure {$D_\gamma u\in \mathscr M(X,H)$ (said weak gradient)
such that for any $\varphi\in \mathcal F C_b^1(X)$ and any $i\in\N$ we have
\[
\int_\Omega u \partial^*_{i} \varphi d\gamma
=-\int_\Omega \varphi d(D_\gamma u)_i,
\]
where $(D_\gamma u)_i=[{D_\gamma u},{h_i}]_H$.} 
If $E\in\mathcal B(X)$ and $u=\1_E$, then we say that $E$ has finite perimeter in $X$ if $u\in BV(X,\gamma)$ and we write $P_\gamma(E,B):=|D_\gamma\1_E|(B)$, for any $B\in\mathcal B(X)$.
\end{definition}

For further informations on $BV(X,\gamma)$ we refer to \cite{AmbMirManPal10}.
\begin{definition}\label{austen}
Let $\Omega\subseteq X$ an open set and {let $p>1$}. We say that $u\in L^{{p}}(\Omega,\gamma)$ is a function with bounded 
variation, $u\in BV(\Omega,\gamma)$, if there
exists a measure {$ D_\gamma u\in \mathscr M(\Omega,H)$} (said weak gradient)
such that for any $\varphi\in {\rm Lip}_{c}(\Omega)$ and any $i\in \N$ we have
\[
\int_\Omega u \partial^*_{i} \varphi d\gamma
=-\int_\Omega \varphi d(D_\gamma u)_i,
\]
where $(D_\gamma u)_i=[{D_\gamma u},{h_i}]_H$.
\end{definition}

\begin{rmk}
In \cite{AmbMirManPal10} $BV(X,\gamma)$ has been defined starting from the Orlicz space $L({\rm Log}L)^{1/2}(X,\gamma)$ instead of $L^p(X,\gamma)$ with $p>1$. Since $L^p(X,\gamma)\subset L({\rm Log}L)^{1/2}(X,\gamma)$ for any $p>1$ Definition \ref{jane} is less general then \cite[Definition 3.1]{AmbMirManPal10}, but in our situation it is enough.\\
Moreover, the same holds for Definition \ref{austen} where $X$ is replaced by the open set $\Omega\subset X$ (see \cite{AdMeMi18}).
\end{rmk}

{It is clear that for $\Omega=X$ the above definitions are equivalent.} {Moreover, if $f\in L^{p}(X,\gamma)$ is a function with bounded variation with weak gradient $D_\gamma u$, clearly 
for every $\Omega$ open subset of $X$,
$f$ is of bounded variation with weak gradient $D_\gamma u \res \Omega$, the restriction of the measure $D_\gamma u$ to the set $\Omega$.}
{In each case, if $D_\gamma u$ exists it is unique.}

\section{Epigraph of Sobolev functions}
\label{crumiro}

Fixed $h\in QX^*$ and an open set $A\subset X^\perp_h$ and
a function $f:A\to \R$, we define the 
graph of $f$ by
\[
\Gamma(f,A):=
\left\{
x=y+f(y)h: y\in A
\right\}
\]
and the epigraph of $f$ by
\[
{\rm Epi}(f,A):=\left\{
x=y+t h: y\in A, t>f(y) 
\right\}.
\]

Let us recall the definition of 
spherical Hausdorff measure in a Wiener space
setting (see \cite{AMP10}, \cite{FD92} and 
\cite{HI10} for more details). For a given $F\subset H$ 
finite dimensional space with $F\subset QX^*$, we define
\[
\mathscr{S}_{F}^{\infty-1}(B)
=\int_{X_F^\perp} \gamma_F^\perp(dy)
\int_{B_y} G_m(z) \mathscr{S}^{m-1}(dz),
\qquad \forall B\in \B(X),
\]
where $m={\rm dim}F$, $\mathscr{S}^{m-1}$ 
is the spherical Hausdorff measure on $F$,
\[
G_m(z)=\frac{1}{(2\pi)^{m/2}} e^{-\frac{|z|^2}{2}}.
\]
and, for any $y\in X^\perp_F$,
\[
B_y=\{z\in F: y+z\in B\}
=(B-y)\cap F.
\]

Since $\mathscr S^{\infty-1}_F\leq \mathscr S^{\infty -1}_G$ if $F\subseteq G$
(see e.g. \cite[Lemma 3.2]{AMP10}, \cite[Proposition 6(ii)]{FD92} or \cite[Proposition 2.4]{HI10}), we can define the measure {
\begin{align*}
\mathscr S^{\infty-1}=\sup_{F}\mathscr S^{\infty-1}_F.
\end{align*}
The definition immediately implies that, If $A\subset X$ is a Borel set which satisfies $\mathscr S^{\infty-1}(A)<+\infty$, then $\gamma(A)=0$. If we now consider an increasing family 
$\mathcal{F}=(F_{n})_{n\in\N}\subset QX^*$ whose closure is dense in $H$, by monotone convergence we have 
that is well defined as a measure 
\[
\mathscr{S}_{\mathcal{F}}^{\infty-1}
=\sup_{n\in \N}
\mathscr{S}^{\infty-1}_{F_n}.
\]
}

From the definition, it follows that $\mathscr S^{\infty-1}_\mathcal F\leq \mathscr S^{\infty-1}$ for any $\mathcal F$ which satisfies 
the above condition. However, the first part of the proof of the following result shows that they coincide if we restrict them to the graph of 
Sobolev functions.

\begin{proposition}\label{nipote}
Let $h\in QX^*$ with $|h|_H=1$, let $A\subseteq X_h^\perp$ be an open set and
let $f\in W^{1,1}(A,\gamma_h^\perp)$. Then:
\begin{itemize}
\item[$(i)$] for any representative $\tilde f$ of $f$ we have $\mathscr S^{\infty-1}(\Gamma(\tilde f,A))<+\infty$. In particular, $\gamma(\Gamma(\tilde f,A))=0$. 
\item[$(ii)$] If $\tilde f_1,\tilde f_2$ are two representatives of $f$, then $\mathscr S^{\infty-1}(\Gamma(\tilde f_1,A)\Delta \Gamma(\tilde f_2,A))=0$ and $\mathscr S^{\infty-1}\res\Gamma(\tilde f_1,A)=\mathscr S^{\infty-1}\res\Gamma(\tilde f_2,A)$.
\item[$(iii)$] For any bounded Borel function $g:X\longrightarrow \R$ we have
\begin{align}
\label{areaformula}
\int_{\Gamma(f,A)} g(x) 
\mathscr{S}^{\infty-1}(dx)
=\int_A g(y+f(y)h) G_1(f(y)) \sqrt{1+|\nabla_H f(y)|_H^2}
\gamma_h^\perp(dy).
\end{align}
\end{itemize}
\end{proposition}

\begin{proof}
Let us then show that for any $B\in\B(X)$ it follows that
\begin{equation}
\mathscr{S}^{\infty-1}
(\Gamma(f,A)\cap B)=
\int_{A} \1_B(y+f(y)h)G_{1}(f(y))
\sqrt{1+|\nabla_{H}f(y)|_{H}^{2}}\ \gamma_{h}^\perp
(dy),
\label{bisnonno}
\end{equation}
where we still denote by $f$ a representative of $f$, since $(i)$, $(ii)$ and $(iii)$ follow from \eqref{bisnonno}. We consider $F\subset QX^*$ 
with $\text{dim}(F)=m<\infty$, $h\in F$,
$\widetilde{F}=X_h^\perp \cap F$, $\pi_{F}$ and $ \pi_{\widetilde{F}}$ canonical
projections of $X$ on $F$ and $\widetilde{F}$, respectively, and we set $X_{F}^{\perp}=\text{ker}(\pi_{F})$ and $X_{\widetilde{F}}^{\perp}=\text{ker}(\pi_{\widetilde{F}})$. This gives $X_h^\perp=\widetilde{F}\oplus X_{F}^{\perp}$ and $F=\widetilde{F}\oplus\langle h\rangle$. Moreover, if we denote by $\gamma_{F}$, $\gamma_{\widetilde{F}}$,
$\gamma_{F}^{\perp}$, $\gamma_{\widetilde{F}}^{\perp}$
the nondegenerate Gaussian measures on $F,\widetilde{F}$, $X_{F}^{\perp}$, $X_{\widetilde{F}}^{\perp}$, respectively, we get
$\gamma=\gamma_{F}\otimes\gamma_{F}^{\perp}$, $\gamma=\gamma_{\widetilde{F}}\otimes\gamma_{\widetilde{F}}^{\perp}$ and $\gamma_h^\perp=\gamma_{\widetilde F}\otimes\gamma_F^\perp$. Then,
\[
\mathscr{S}_{F}^{\infty-1}({\Gamma(f,A)\cap B})=\int_{X_{F}^{\perp}}\gamma_{F}^{\perp}(dy)\int_{{(\Gamma(f,A)\cap B)}_{y}}G_{m}(z)\ \mathscr{S}^{m-1}(dz),
\]
where $G_{m}(z):=\frac{1}{\sqrt{{(2\pi)^{m}}}}\exp-\frac{\left\Vert z\right\Vert {}_{F}^{2}}{2}$
for $z\in F$ and
\[
(\Gamma(f,A)\cap B)_{y}=
\{z\in F|z+y\in\Gamma(f,A)\cap B\},
\]
for all $y\in X_{F}^{\perp}$.
For any $y\in X_{F}^\perp$, the set $A_y$, which a priori is contained in $F$, is indeed contained in $\widetilde F$ since $A\subset X_h^\perp$. We consider the function $f_y:A_y\longrightarrow\R$ defined by $f_y(z):=f(y+z)$. Since $f\in W^{1,1}(A,\gamma_h^\perp)$, it follows that $f_y\in W^{1,1}({A_y},\gamma_{{ F}})$ for $\gamma_{{\widetilde F}}^\perp$-a.e. $y\in X_{{F}}^\perp$. Let us denote by $\Gamma(f_{y},A_y)\subseteq F$  the graph of $f_{y}$ on $A_y$. Since $(\Gamma(f,A))_y=\Gamma(f_y,A_y)$ and
\[
(\Gamma(f,A)\cap B)_{y}=\Gamma(f_{y},A_y) \cap B_y.
\]
Therefore, writing $z\in\Gamma(f_{y},A_y)$ as $z=\tilde z+[z,h]_H h$ with $\tilde z\in \widetilde F$, we get
\[
G_{m}(z)=G_{m}(\tilde{z}+f_{y}(\tilde{z})h)=G_{m-1}(\tilde{z})G_{1}(f_{y}(\tilde{z}))
\]

Since $f_y$ is a finite--dimensional Sobolev function, it follows that
\begin{align*}
\int_{\Gamma({f}_y, A_{y})} &
\1_{B}(y+z)G_{m}(z)\ 
\mathscr{S}^{m-1}(dz)\\
=&
\int_{A_y}
\1_B(y+\tilde z+f_y(\tilde z)h)
G_{m-1}(\tilde{z})G_{1}(f_{y}(\tilde{z}))
\sqrt{1+|\nabla_{F} f_{y}(\tilde{z})|^{2}}\ d\tilde{z}\\
=&
\int_{A_y}
\1_B(y+\tilde z+f_y(\tilde z)h)
G_1(f_y(\tilde z)) 
\sqrt{1+|\nabla_{F} f_{y}(\tilde{z})|^{2}}
\gamma_{\widetilde F}(d\tilde z).
\end{align*}
Hence,
\begin{align*}
\mathscr{S}_{F}^{\infty-1}(\Gamma(f,A)\cap B)
=&\int_{X_{F}^{\perp}}\gamma_{F}^{\perp}(dy)\int_{A_y}
\1_B(y+\tilde z+f_y(\tilde z)h)
G_{1}(f_{y}(\tilde{z}))\sqrt{1+|\nabla_{F} f_{y}(\tilde{z})|_H^{2}}\ \gamma_{\widetilde{F}}(d\tilde{z})\\
=&
\int_{A}
\1_B(y+f(y)h)
G_{1}(f(y))\sqrt{1+|\pi_{F}(\nabla_{H}f(y))|_{H}^{2}}\ \gamma_{h}^\perp (dy) \\
\leq
&
\int_{A}
\1_B(y+f(y)h)
G_{1}(f(y))\sqrt{1+|\nabla_{H}f(y)|_{H}^{2}}\ \gamma_{h}^\perp (dy).
\end{align*}
Therefore,
\begin{align*}
\mathscr S^{\infty-1}(\Gamma(f,A)\cap B)
\leq
&
\int_{A}
\1_B(y+f(y)h)
G_{1}(f(y))\sqrt{1+|\nabla_{H}f(y)|_{H}^{2}}\ \gamma_{h}^\perp (dy).
\end{align*}

If we now consider an increasing family 
$\mathcal{F}=(F_{n})_{n\in\N}\subset QX^*$ whose closure is dense in $H$ {and $h\in F_1$, by
monotone convergence we obtain that 
\[
\mathscr{S}_{\mathcal{F}}^{\infty-1} (\Gamma(f,A)\cap B)
=\sup_{n\in \N}
\mathscr{S}^{\infty-1}_{F_n}(\Gamma(f,A)\cap B)
=\int_{A} \1_B\sqrt{1+|\nabla_H f(y)|_H^2} {\gamma_h^\perp}(dy).
\]
Hence we have  
\begin{align}
\mathscr{S}^{\infty-1} (\Gamma(f,A)\cap B)
=\mathscr S_{\mathcal F}^{\infty-1}(\Gamma(f,A)\cap B)= & \int_{A} \1_B(y+f(y)h)\sqrt{1+|\nabla_H f(y)|_H^2} {\gamma_h^\perp}(dy)
\label{prozio}
\end{align}
}

\vspace{2mm}
{\bf Proof of $(i)$}. If we take $B=X$ in \eqref{prozio}, then we have
\begin{align*}
\mathscr S^{\infty-1}(\Gamma(f,A))
= & \int_{A} \sqrt{1+|\nabla_H f(y)|_H^2} {\gamma_h^\perp}(dy)
\leq \int_A\left(1+|\nabla_Hf(y)|_H\right)\gamma_h^\perp(y) \\
\leq & \gamma_h^\perp(A)+\|f\|_{W^{1,1}(A,\gamma_h^\perp)}<+\infty.
\end{align*}

\vspace{2mm}
{\bf Proof of $(ii)$}. Let $\tilde f_1$ and $\tilde f_2$ be two representatives of $f$. Let us set $N:=\{y\in A:\tilde f_1(y)\neq\tilde f_2(y)\}$. Then, $\gamma_h^\perp(N)=0$ and it is easy to see that if $x=y+\tilde f_1(y)h\in \Gamma(\tilde f_1,A)\setminus \Gamma(\tilde f_2,A)$, then $y\in N$. Therefore, from \eqref{prozio} with $f$ replaced by $\tilde f_1$ and $B$ replaced by $\Gamma(\tilde f_1,A)\setminus \Gamma(\tilde f_2,A)$ we deduce that
\begin{align*}
\mathscr S^{\infty-1}(\Gamma(\tilde f_1,A)\setminus \Gamma(\tilde f_2,A))
= & \int_A\1_{(\Gamma(\tilde f_1,A)\setminus \Gamma(\tilde f_2,A))}(y+\tilde f_1(y)h)\sqrt{1+|\nabla_H \tilde f_1(y)|_H^2} {\gamma_h^\perp}(dy) \\
\leq & \int_A\1_{N}(y)\sqrt{1+|\nabla_H \tilde f_1(y)|_H^2} {\gamma_h^\perp}(dy)=0.
\end{align*}
The same arguments give $\mathscr S^{\infty-1}(\Gamma(\tilde f_2,A)\setminus \Gamma(\tilde f_1,A))=0$, and we get the first part of the statement. As far the second part is concerned, it is enough to notice that for any Borel set $B\in \B(X)$ we have
\begin{align*}
\mathscr S^{\infty-1}\res\Gamma(\tilde f_1,A)(B)
= & 
\mathscr{S}^{\infty-1} (\Gamma(\tilde f_2,A)\cap B)
= \int_{A} \1_B(y+f(y)h)\sqrt{1+|\nabla_H \tilde f_1(y)|_H^2} {\gamma_h^\perp}(dy) \\
= & \int_{A} \1_B(y+f(y)h)\sqrt{1+|\nabla_H \tilde f_2(y)|_H^2} {\gamma_h^\perp}(dy) 
=  \mathscr{S}^{\infty-1} (\Gamma(\tilde f_2,A)\cap B)\\
= & \mathscr S^{\infty-1}\res\Gamma(\tilde f_2,A)(B).
\end{align*}

\vspace{2mm}
{\bf Proof of $(iii)$}. \eqref{areaformula} follows from \eqref{prozio} simply approximating $g$ by means of simple functions.
\end{proof}

\begin{remark}\label{giampaolo}
\item[$(i)$]
Assume that the function $g$ is \eqref{areaformula} does not depend on the component of $x$ along $h$, i.e., there exists a Borel function $\ell:A\longrightarrow \R$ such that $g(x)=\ell(y)$, where $y=x-\pi_hx$. Then, if $\tilde g$ is a bounded Borel function such that $\tilde g(x)=g(x)$ for $\gamma$-a.e. $x\in X$, then
\begin{align*}
\int_{\Gamma(f,A)} g 
d\mathscr{S}^{\infty-1}=\int_{\Gamma(f,A)} \tilde g d\mathscr{S}^{\infty-1}.
\end{align*}
Indeed, for $\gamma$-a.e. $x\in X$ we have $\tilde g(x)=\ell(y)$, with $y=x-\pi_hx$, and therefore
\begin{align*}
\int_{\Gamma(f,A)} g 
d\mathscr{S}^{\infty-1}
= & \int_Ag(y+f(y)h)\sqrt {1+|\nabla_Hf(y)|^2}\gamma_h^\perp(dy)
=\int_A\ell(y)\sqrt {1+|\nabla_Hf(y)|^2}\gamma_h^\perp(dy) \\
= & \int_A\tilde g(y+f(y)h)\sqrt {1+|\nabla_Hf(y)|^2}\gamma_h^\perp(dy)
= \int_{\Gamma(f,A)} \tilde g d\mathscr{S}^{\infty-1}.
\end{align*}
\end{remark}

\begin{theorem}\label{bicicletta}
Let $h\in QX^*$ with $|h|_H=1$, let $A\subseteq X_h^\perp$ be an open set and let $f$ be a Borel representative of an element of 
$W^{1,1}(A,\gamma_h^\perp)$. 
Then, the {Borel} set
\[
Epi(f,A)=\left\{
x=y+th: y\in A, t>f(y)
\right\}
\]
has finite perimeter in the cylinder
$C_A=A\oplus \langle h\rangle$
with
\begin{align}
D_\gamma \1_{Epi(f,A)}(dx)
=\frac{-\nabla_H f(y)+h}{\sqrt{1+|\nabla_H f(y)|_H^2}}
\mathscr{S}^{\infty-1}\res \Gamma(f,A)(dx),
\label{fedor}
\end{align}
where $x=y+f(y)h$.
As a byproduct, we get
\[
P({\rm Epi}(f,A),C_A)=
\int_A G_1( f(y))
\sqrt{1+|\nabla_H f(y)|^2_H} \gamma_h^\perp (dy).
\
\]
\end{theorem}

\begin{proof}
At first, we stress that from Proposition \ref{nipote}$(ii)$ and Remark \ref{giampaolo} formula \eqref{fedor} does not depend on the choice of the representative $f$.  
Let us denote by $\nu_f$ the vector defined
on $C_A$ by
\begin{align}
\nu_f(x)=\frac{-\nabla_H f(y)+h}{\sqrt{1+|\nabla_H f(y)|_H^2}},
\label{ricola}
\end{align}
where $x=y+th$ with $y\in A$ and $t\in\R$. First of all, we notice that for $\varphi\in 
{\rm Lip}_{c}(C_A)$ we have that
\begin{align*}
\int_{{\rm Epi}(f,A)} \partial^*_h \varphi(x)
\gamma(dx)
=&
\int_A \gamma_h^\perp (dy)
\int_{f(y)}^\infty \partial^* \varphi_y(t) \gamma_1(dt)
\\
=&
-\int_A G_1(f(y))\varphi(y+f(y)h) \gamma_h^\perp(dy)\\
=&
-\int_A G_1(f(y))
\frac{\varphi(y+f(y)h)}{\sqrt{1+|\nabla_H f(y)|_H^2}} 
\sqrt{1+|\nabla_H f(y)|_H^2}\gamma_h^\perp(dy)\\
=&
-\int_{\Gamma(f,A)} \varphi [{\nu_f},{h}]_H 
d\mathscr{S}^{\infty-1}.
\end{align*}
In the last equality we have applied \eqref{areaformula} with $g=\nu_f$ on $C_A$ and $g= 0$ on $(C_A)^c$, noticing that $g(x)=\ell(x-\pi_{h}x)$ with $x\in C_A$. 

Let us now fix $k\in h^\perp$, $k\in Q X^*$ and
we consider $W=\text{ker}(\pi_h)\cap\text{ker}(\pi_k)$;
we have $X=W\oplus \langle h,k\rangle$ and 
$\gamma=\gamma_W\otimes \gamma_{\langle h,k\rangle}$.
We notice that for $\gamma_W$-a.e. $w\in W$
\[
({\rm Epi}(f,A))_w
=\{ z_1h+z_2k\in \langle h,k\rangle : z_1> f_w(z_2k), z_2 k\in A_w\},
\quad A_w=\{z_2k: w+z_2k \in A\},
\]
and the map $f_w:A_w\longrightarrow \R$
belongs to $W^{1,1}(A_w,\gamma_W)$.
Then, the set ${\rm Epi}(f_w,A_w)$ has finite perimeter for $\gamma_W$-a.e. $w\in W$ with
bounded inner normal given by $\nu_w=(-f'_w,1)/\sqrt{1+(f'_w)^2}$. For any $\varphi\in 
{\rm Lip}_{c}(C_A)$ we get
\begin{align*}
\int_{{\rm Epi}(f,A)} \partial^*_k \varphi(x)\gamma(dx)
=&
\int_W \gamma_W(dw)
\int_{{\rm Epi}(f_w,A_w)}
\partial^*_2 \varphi_w(z) 
\gamma_{\langle h,k\rangle}(dz)\\
=&
\int_W \gamma_W(dw)
\int_{\Gamma(f_w,A_w)} \frac{f_w'}{\sqrt{1+(f'_w)^2}} 
\varphi_w G_1(f_w)d\mathscr{S}^1\\
=&
\int_W \gamma_W(dw)
\int_{A_w} f_w'(z_2) \varphi(w+f_{{w}}(z_2)h+z_2k)
G_1(f_{{w}}(z_2))\gamma_1(dz_2)\\
=&
\int_A \partial_k f(y) \varphi(y+f(y)h)
G_1(f(y)) \gamma_h^\perp(dy)
\\
=&
-\int_A  \varphi({y+f(y)h}) [{\nu_f(y+f(y) h)},{k}]_H
G_1(f(y)) \sqrt{1+|\nabla_H f(y)|_H^2} 
\gamma_h^\perp(dy)\\
=&
-\int_{\Gamma(f,A)} \varphi [{\nu_f},{k}]_H
d\mathscr{S}^{\infty-1},
\end{align*}
where we have used the fact that
${\rm ker}(\pi_h)=W + {\langle k\rangle}$ and that $\gamma_h^\perp =\gamma_W\otimes \gamma_k$.
Let us consider an orthonormal basis $\{h,h_n:n\in\N\}\subset QX^*$ of $H$. 
Then, we have proved that for any $\varphi\in {\rm Lip}_{c,b}(C_A,H)$ and any $k\in\{h,h_n:h\in\N\}$,
\[
\int_{{\rm Epi}(f,A)} \partial_k^* \varphi(x)\gamma(dx)
=
-\int_{\Gamma(f,A)} {\varphi}[{\nu_f},k]_H \mathscr{S}^{\infty-1},
\]
i.e. the measure 
\[
\mu= \nu_f \mathscr{S}^{\infty-1}_{\mathcal{F}}
\res \Gamma(f,A)\in \mathscr{M}(C_A,H)
\]
is the distributional derivative of $\1_{{\rm Epi}(f,A)}$. Finally, Proposition \ref{nipote}$(i)$ implies that ${\rm Epi}(f,A)$ has finite perimeter in $C_A$.
\end{proof}

We conclude this section providing a useful result on epigraphs of convex and concave functions.
\begin{remark}
\label{business}
Let $h\in QX^*$ with $|h|_H=1$, let $D\subset X_h^\perp$ be an open convex domain,  let $g$ be a continuous convex function and let $f$ be a continuous concave function both defined on $D$. Then,
${\rm  Epi}(g,D)$ and $C_D\setminus \overline{{\rm Epi}(f,D)}$ are open convex subsets of $X$, and therefore ${\rm  Epi}(g,D)$ and ${\rm  Epi}(f,D)$ have finite perimeter in $X$.
Indeed, since a function is convex if and only if its epigraph is convex, from \cite[Proposition 9]{CLMN12} it follows that ${\rm Epi}(g,D)$ is convex.
Analogously, $C_D$ and $C_D\setminus \overline{{\rm Epi}(f,D)}$ have finite perimeter in $X$ since they are open convex sets. This implies that also $X\setminus {\rm Epi}(f,D)$ has finite perimeter in $C_D$, and therefore ${\rm Epi}(f,D)$ has finite perimeter.
\end{remark}

%
%

\section{Integration by parts formula on convex sets}
\label{coppadelmondo}

In this section we consider a nonempty open convex set $\Omega
\subset X$. By \cite[Proposition 9]{CLMN12}, $\Omega$ has finite perimeter
in $X$ and $\gamma(\partial\Omega)=0$, i.e. $\1_\Omega\in BV(X,\gamma)$. 
Without loss of generality we can assume that $0\in \Omega$, and we define
\[
\Omega{=}\{ \p<1\},
\]
with $\p$ being the gauge of the convex set or the Minkowski functional associated with $\Omega$ centered in $0$ defined by
\[
\p(x)=\inf \{ \lambda>0: x \in \lambda\Omega\}.
\]
The main result proved in this section is the following theorem.
\begin{theorem}
$\nabla_H\p$ is defined $\mathscr S^{\infty-1}$-almost everywhere and non-zero on 
$\partial\Omega$, for any $k\in H$ and any $\psi\in {\rm Lip}_{b}(X)$ we have that 
\begin{align*}
\int_\Omega\partial^*_k\psi d\gamma
=\int_{\partial\Omega} \psi\frac{\partial_k\p}{|\nabla_H\p|_H}d\mathscr S^{\infty-1}.
\end{align*}
\label{triciclo}
\end{theorem}
The proof of Theorem \ref{triciclo} is postponed to the end of the section.

\medskip{}

{Let us introduce some useful tools about convex functions (we refer to \cite[Chapter 5]{Phe93} for further details).
We consider the dual ball of $\p$ defined by
\begin{align*}
C(\p):=\{x^*\in X^*:\langle x^*,x\rangle \leq \p(x) \ \forall x\in X\},
\end{align*}
Moreover, we recall that, for any $x_0\in X$, the subdifferential $\partial \p(x_0)$ is the set of the elements $x^*\in X^*$ which satisfy
\begin{align*}
x^*(x-x_0) \leq \p(x)-\p(x_0), \quad \forall x\in B(x_0,r),
\end{align*}
for some $r>0$.
We will use the following property of the subdifferential of a convex function.
\begin{proposition}{\cite[Proposition 1.11]{Phe93}}
\label{stereo}
Let $f$ be a convex function which is continuous at $x_0\in D$, where $D$ is a convex domain. Then, $\partial f(x_0)$ is nonempty.
\end{proposition}
\medskip{}}

{We state the following characterization of the subdifferential $\partial \p(x)$ of a Minkowski functional in terms of $C(\p)$ (see \cite[Lemma 5.10]{Phe93}).
\begin{lemma}\label{pasta}
$x^*\in \partial \p(x)$ if and only if $x^*\in C(\p)$ and $x^*(x) =\p(x)$. \end{lemma}}

{In our case,} thanks to \cite[Lemma 6.2]{AMP10} we may consider $h\in QX^*$ such that
\begin{align}
|D_\gamma\1_\Omega|(\{ [{\nu_\Omega},{h}]_H=0\})=0,
\label{sinonimo}
\end{align}
with 
$D_\gamma \1_\Omega =\nu_\Omega |D_\gamma \1_\Omega|$.
This Lemma simply says that we may choose a direction
$h$ such that the vertical part of $\partial \Omega$
with respect to $h$ is $|D_\gamma\1_\Omega|$-negligible. We denote by $h^*$ the element of $X^*$ such that $h=Qh^*$.
Once such a direction has been fixed, 
we may define the open convex set $\Omega_h^\perp \subseteq X_h^\perp$ by
\[
\Omega_h^\perp =\{ y\in X^\perp_h : \exists
t\in \R \mbox{ s.t. } y+th\in \Omega \}.
\]
For any $y\in \Omega_h^\perp$, the set 
\[
\Omega_y =\{ t\in \R: y+th \in \Omega\} 
\]
is an open interval, and therefore there exist 
$f:\Omega^\perp_h\to (-\infty,+\infty]$, $g:\Omega^\perp_h\to 
[-\infty,+\infty)$ such that $\Omega_y$ is the interval $(g(y),f(y))$,
{i.e., $\Omega$ is between the graph of $g$ and that of $f$.}

\begin{lemma}\label{nevicata}
The functions $f$ and $g$ satisfy the following properties:
\begin{itemize}
\item[$(i)$] If there exists $y\in\Omega_h^\perp$ such that $f(y)=+\infty$, then $f\equiv+\infty$ on $\Omega_h^\perp$. 
Analogously, if $g(y)=-\infty$ for some $y\in\Omega_h^\perp$, then $g\equiv-\infty$ on $\Omega_h^\perp$.
\item[$(ii)$]if $f$ is not infinite then it is a concave function. Analogously, if  $g$ in not infinite then it is a convex function.\end{itemize}
\end{lemma}
\begin{proof}
{To show $(i)$, let us assume that there exists 
$y_0\in \Omega_h^\perp$
such that $f(y_0)=+\infty$, and let $y\in \Omega_h^\perp$. Therefore, there exists $y_1\in \Omega_h^\perp$ and $\lambda>0$ s.t. $y=\lambda y_0 + (1-\lambda)y_1$.
From the definition of $\Omega_h^\perp$ there exists $t_1\in\R$ s.t. $x_1=y_1+t_1h\in\Omega$, and since $f(y_0)=+\infty$ we have $x_0=y_0+t h\in \Omega$ for every $t\in(g(y_0),+\infty)$. Since $\Omega$ is convex, we have $\lambda x_0+(1-\lambda)x_1=y+(\lambda t+(1-\lambda)t_1)h\in\Omega$ and therefore
\begin{align*}
f(y)\geq \lambda t+(1-\lambda)t_1, \quad t\in(g(y),+\infty),
\end{align*}
which gives $f(y)=+\infty$.}

{The same argument holds for $g$, i.e., if there 
exists $y_0\in \Omega^\perp_h$ such that
$g(y_0)=-\infty$, then $g\equiv -\infty$.}

\vspace{2mm}
Let us prove $(ii)$. Assume that $g>-\infty$ on $\Omega_h^\perp$. We fix $y_1,y_2 
\in \Omega_h^\perp$, $t_1\in \Omega_{y_1}$,
$t_2\in \Omega_{y_2}$, then for any $\lambda\in[0,1]$
\[
\lambda (y_1+t_1 h)+(1-\lambda)(y_2+t_2h)
=\lambda y_1+(1-\lambda) y_2+
(\lambda t_1 +(1-\lambda)t_2)h
\in \Omega.
\]
This means that $\tilde y:=\lambda y_1+(1-\lambda)y_2\in\Omega_h^\perp$ and $\lambda t_1+(1-\lambda)t_2\in\Omega_{\tilde y}$. Therefore, 
\[
g(\lambda y_1+(1-\lambda) y_2) \leq 
\lambda t_1 +(1-\lambda)t_2
{\leq}
f(\lambda y_1+(1-\lambda) y_2).
\]
Since this is true for any $t_1$ and $t_2$, this 
implies that
\[
g(\lambda y_1+(1-\lambda) y_2)\leq
\lambda g(y_1)+(1-\lambda)g(y_2)
\]
hence $g$ is convex. same arguments reveal that for any $\lambda\in[0,1]$ we have
\[
\lambda f(y_1)+(1-\lambda) f(y_2)
\leq
f(\lambda y_1+(1-\lambda) y_2),
\]
which implies that $f$ is concave.
\end{proof}

Thanks to Lemma \ref{nevicata} (and by the fact that $\Omega$ is a nonempty set, hence it is 
impossible that $f=-\infty$ everywhere or $g=+\infty$ everywhere), only the following four cases occur:
\begin{enumerate}
\item $f\equiv +\infty$, $g\equiv -\infty$ and 
$\Omega_h^\perp =X^\perp_h$, i.e. $\Omega=\Omega_h^\perp\oplus \langle h\rangle$;
\item $f\equiv +\infty$ and $g(y)>-\infty$ for 
any $y\in \Omega_h^\perp$, and then
$\Omega={\rm Epi}(g,\Omega_h^\perp)$ .
\item $g\equiv -\infty$ and $f(y)<+\infty$ for 
any $y\in \Omega_h^\perp$, and then
$\Omega=\{ x=y+th: y\in \Omega_h^\perp, t<f(y)\}$.
\item $-\infty< g(y)<f(y)<+\infty$  for 
any $y\in \Omega_h^\perp$ and
\[
\Omega=\{ x=y+th: y\in \Omega_h^\perp, 
t\in (g(y),f(y))\}.
\]
\end{enumerate}
From now on we shall assume to be in the last case, since in the other three cases the following
lemmas remain true, with the convention that $\Gamma(f,\Omega_h^\perp)=\varnothing$ if $f=+\infty$, and  
 $\Gamma(g,\Omega_h^\perp)=\varnothing$ if $g=-\infty$.
 
Before passing to the infinite dimension, we state a property of open convex sets in finite dimension. 
\begin{remark}
\label{camilleri}
For any open convex set $C\subset \R^n$, $\partial^* C=\partial C$, i.e., each point of $\partial C$ has density different from $0$ and $1$. Let $n=2$, let us fix $x\in\partial C$ and {let $\p$ be its Minkowski function. By Proposition \ref{stereo}, there exists $\nu\in\partial \p(x) $, hence
$C$ remains below the hyperplane with equation $\langle \nu,\cdot -x\rangle=0$ (it suffices to remember the definition of $\partial \p$
and the fact that if $y\in C$ then $\p(y)<1$, while $\p(x)=1$}), 
which implies that its density is not greater than $1/2$. Further, let $B(x_0,r)\subset C$. 
The convexity of $C$ implies that the convex hull of $\{x,B(x_0,r)\}$ is contained in $C$, 
and in particular the triangle with vertices $x,x_0$ and $x_1$, where $x_1$ satisfies $|x_1-x_0|=r$ and $x_1-x_0\perp x-x_0$. 
Therefore, for any $\rho>0$ a sector of angle $2\arctan(r|x-x_0|^{-1})$ of the ball $B(x,\rho)$ is contained in $C$. This gives
\begin{align*}
\frac{|C\cap B(x,\rho)|}{|B(x,\rho)|}\geq 2{\rm arctg}\left(\frac{r}{|x-x_0|}\right)>0,
\end{align*} 
for any $\rho>0$, and so the density of $x$ is greater than $0$. The general case $n\in\N$ follows from similar arguments.
\end{remark}

Let $\mathcal F$ be a countable family of finite dimensional subspaces of $QX^*$ stable under finite union and such that $\cup_{F\in\mathcal F}F$ 
is dense in $H$. In \cite{HI10} the $\mathcal F$-essential boundary of $\Omega$ is defined by
\begin{align*}
\partial^*_{\mathcal F}\Omega=\bigcup_{F\in\mathcal F}\bigcap_{G\supset F, G\in\mathcal F}\partial^*_G\Omega,
\end{align*}
where
\begin{align*}
\partial_F^*\Omega:=\{y+z:y\in {\rm Ker}(\pi_F), \ z\in\partial^*(\Omega_y)\},
\end{align*}
for any $F\in\mathcal F$. In general, $\partial_F^*\Omega$ does not satisfy any monotonicity property with respect to $F\in\mathcal F$. 
However, in the case of open convex sets we recover the finite dimensional situation with the next Lemma.

\begin{lemma}\label{salama}
Let $\Omega\subset X$ be an open convex set and let $\mathcal F$ be as above.
Then, $\partial ^*_{\mathcal F}\Omega=\partial\Omega$.
\end{lemma}

\begin{proof}
{At first, we claim that $\partial^*_{F}\Omega\subseteq \partial^*_G\Omega$ if $F\subseteq G$, for any $F,G\in\mathcal F$. 
Let $F\in\mathcal F$ and let $y+z\in \partial^*_{F}\Omega$. This means that $y\in{\rm Ker}(\pi_{F})$ and $z\in\partial(\Omega_y)$ 
($\partial(\Omega_y)=\partial^*(\Omega_y)$ since it is convex, see Remark \ref{camilleri}). 
Let $G\in\mathcal F$ be such that $F\subseteq G$. In particular, there exists a finite dimensional subspace $L$ of $QX^*$ such that $G=F \oplus L$. 
If $L=\{0\}$, we are done. Assume that $L\neq \{0\}$. Therefore, $y+z=y-\pi_L y+\pi_{L}y+z=:\tilde y+\tilde z$, 
where $\tilde y=y-\pi_L y$ and $\tilde z:=\pi_Ly+z$.
Clearly, $\tilde y\in {\rm Ker}(\pi_{G})$ and $\tilde z\in G$. It remains to prove that $\tilde z\in\partial^*(\Omega_{\tilde y})$. 
Since $\Omega_{\tilde y}$ is a finite dimensional open convex set, from Remark \ref{camilleri} it is equivalent to show that 
$\tilde z\in\partial(\Omega_{\tilde y})$. By contradiction, we suppose that $\tilde z\in \Omega_{\tilde y}$. 
Then, $y+z=\tilde y+\tilde z\in \Omega$, and so $z\in \Omega_y$. This contradicts the assumptions, since $\Omega_y$ is open 
and $z\in\partial^*(\Omega_y)=\partial(\Omega_y)$. Moreover, $\tilde z\in \overline{\Omega_{\tilde y}}$. Indeed, since $z\in\partial (\Omega_y)$, 
there exists a sequence $(z_n)\in \Omega_y$ which converges to $z$ in $X$. Obviously, the sequence $(\tilde z_n:=\pi_L(y)+z_n)$ converges to $\tilde z$ 
in $X$ and $\tilde y+\tilde z_n=y+z_n\in\Omega$, which means that $\tilde z\in(\overline{\Omega_{\tilde y}})$. 
Hence, $\tilde z\in\partial(\Omega_{\tilde y})=\partial^*(\Omega_{\tilde y})$, since $\Omega_{\tilde y}$ is convex.
This means that $\tilde y+\tilde z\in\partial^*_{G}(\Omega)$, and the claim is therefore proved.
}

In particular, the claim implies that $\partial_{\mathcal F}^*\Omega=\cup_{F\in\mathcal F}\partial_{F}^*\Omega$.
We remark that $\cup_{F\in\mathcal F}F$ is dense in $X$. This fact easily follows from the density of $\cup_{F\in\mathcal F}F$ in $H$, 
the density of $H$ in $X$ and the continuous embedding $H\hookrightarrow X$. We stress that, for any $F\in \mathcal F$ and any $y\in{\rm Ker}(\pi_{F})$, 
arguing as above we deduce that $\partial(\Omega_y)\subset(\partial \Omega)_y$. Hence, $\partial^*_{\mathcal F}\Omega\subset \partial \Omega$.
To show the converse inclusion, we consider $x\in \partial \Omega$. Since $\Omega$ is open, there exists an open ball $B\subset \Omega$. 
Clearly, $\tilde B:=x-B$ is an open ball in $X$, and the density of $\cup_{F\in\mathcal F}F$ in $X$ implies that there exists $F\in\mathcal F$ and $\xi\in F$ 
such that $\xi=x-y$, with $y\in B$, i.e., $x=y+\xi$. If we define $y_F=y-\pi_{F}y\in{\rm Ker}(\pi_{F})$ 
and $z_F:=\pi_{F}y+\xi$, it remains to prove that $z_F\in\partial (\Omega_{y_F})=\partial^*(\Omega_{y_F})$. Clearly, $z_F\notin\Omega_{y_F}$, 
otherwise $x=y_F+z_F\in\Omega$. Further, since $y\in \Omega$ and $x\in\partial \Omega$, for any $\lambda\in[0,1)$
we have $y+\lambda\xi=y+\lambda(x-y)\in\Omega$. 
Taking a sequence $(\lambda_m)_m\in\N\subset (0,1)$ converging to $1$, we obtain a sequence $(\eta_m=\lambda_mz_F)_{m\in\N}\subset\Omega_{y_F}$ 
which converges to $z_F$ in $F$, and so $z_F\in\overline {\Omega_{y_F}}$, which gives $x\in\partial^*_F\Omega$.
\end{proof}

\begin{remark}\label{amarena}
From \cite{AMP10} and \cite{HI10}, we know that for any $B\in\mathcal B(X)$ 
we have $|D_\gamma\1_\Omega|(B)=\mathscr S^{\infty-1}_{\mathcal F}(B\cap \partial_{\mathcal F}^*\Omega)$, 
for any countable family $\mathscr F$ of finite dimensional subspaces of $QX^*$ stable under finite union such that 
$\cup_{F\in\mathcal F} F$ is dense in $H$. In particular, if $\mathcal F'$ satisfies the same assumptions as $\mathcal F$, 
then from Lemma \ref{salama} we deduce that 
$\mathscr S^{\infty-1}_{\mathcal F}(B\cap\partial\Omega)=\mathscr S^{\infty-1}_{\mathcal F'}(B\cap \partial\Omega)$. 
Therefore, $\mathscr S^{\infty-1}_{\mathcal F}\res\partial\Omega=\mathscr S^{\infty-1}_{\mathcal F'}\res\partial\Omega$ for any $\mathcal F,\mathcal F'$ 
as above and from the definition of $\mathscr S^{\infty-1}$ we infer that 
$\mathscr S^{\infty-1}_\mathcal F\res{\partial \Omega}=\mathscr S^{\infty-1}_{\mathcal F'}\res{\partial \Omega}=
\mathscr S^{\infty-1}\res\partial\Omega$. In particular, we get $|D_\gamma\1_\Omega|(B)=\mathscr S^{\infty-1}(B\cap\partial\Omega)$ 
for any $B\in\mathcal B(X)$. 
\end{remark}

\begin{lemma}\label{cementine}
Let $\Omega$ be an open convex set, $h\in QX^*$, $C:=\Omega_h^\perp\oplus \langle h\rangle$ and 
let $f,g$ be the functions introduced in Lemma \ref{nevicata}. Then
\begin{align}\label{cubo}
\partial \Omega=\Gamma(f,\Omega_h^\perp)\cup\Gamma(g,\Omega_h^\perp)\cup N,
\end{align}
where the sets in the right-hand side of \eqref{cubo} are pairwise disjoint, and $\mathscr S^{\infty-1}(N)=0$. In particular,
\[
\mathscr{S}^{\infty-1}
(\partial \Omega\setminus (\Gamma(f,\Omega_h^\perp)\cup
\Gamma(g,\Omega_h^\perp)))=0.
\]
\end{lemma}
\begin{proof}
Since
$C$ is convex, from Remark \ref{amarena} it follows that $D_\gamma \1_C=\nu_C \mathscr S^{\infty-1}\res\partial C $.
Further, $\partial \Omega=(\partial\Omega\cap C)\cup N$. Since $\Omega\subset C$, we have $N=\partial \Omega\cap \partial C$, and by 
\cite[Corollary 2.3]{AMP15} $\nu_\Omega(x)=\nu_C(x)$ for $\mathscr S^{\infty-1}$-a.e. $x\in \partial\Omega\cap\partial C$. By construction,
$[ \nu_C(x),h]_H=0$ for $\mathscr S^{\infty-1}$-a.e. $x\in\partial C$, and so $[\nu_\Omega(x),h]_H=0$ for $\mathscr S^{\infty-1}$-a.e. $x\in N$. 
Therefore, \eqref{sinonimo} gives $|D_\gamma\1_{\Omega}|(N)=0$, and since $N\subset \partial\Omega$, from Remark \ref{amarena} 
we deduce that $\mathscr S^{\infty-1}(N)=\mathscr S^{\infty-1}(N\cap \partial\Omega)=|D_\gamma\1_{\Omega}|(N)=0$.
 
It remains to show that $\partial \Omega\cap C=\Gamma(g,\Omega_h^\perp) \cup \Gamma(f,\Omega_h^\perp)$.
At first, we suppose that $x\in \Gamma(g,\Omega_h^\perp)$. Hence, there exists $y\in \Omega_h^\perp$ such that $x=y+g(y)h$. Arguing as above, we deduce that $x\in \overline {\Omega}\setminus \Omega=\partial \Omega$, and clearly $x\in C$. Further, the same arguments hold true for $x\in\Gamma(f,\Omega_h^\perp)$. Inclusion $\supseteq$ is therefore proved.

To show the converse inclusion, we assume that $x\in \partial\Omega\cap C$. Therefore, there exists $\delta\in\R$ such that $x+\delta h\in\Omega$. 
Let us assume that $\delta>0$. If we set $y:=(I-\pi_h)x\in\Omega_h^\perp$ and $z:=\pi_hx$, then $y+z+th\in\Omega$ 
for any $t\in(0,\delta)$ (because $\Omega$ is convex), i.e., $z+th\in\Omega_y$ for any $t\in(0,\delta)$. Letting $t\rightarrow0$, we get that $\pi_hx\in \overline{\Omega_y}$. 
Necessarily, $\pi_hx\notin\Omega_y$, otherwise $x=y+z\in\Omega$, which contradicts the fact that $x\in\partial \Omega$. 
Hence, $z\in\partial (\Omega_y)=\{g(y),f(y)\}$ and, since $\delta>0$, we deduce that $z=g(y)$ which means $x=y+g(y)h\in\Gamma(g,\Omega_h^\perp)$. 
If $\delta<0$, arguing as above we infer that $z=f(y)$, from which it follows that $x=y+f(y)h\in\Gamma(f,\Omega_h^\perp)$. 
\end{proof}

\begin{lemma}
\label{james}
For any $y_0\in \Omega_h^\perp$, there exists $r_0=r_0(y_0)>0$ such that $f,g$ are bounded Lipschitz functions on $B(y_0,r_0)\cap X_h^\perp$. As a byproduct, $f$ and $g$ are G\^ateaux differentiable $\gamma_h^\perp$-a.e. $\in B(y_0,r_0)\cap X_h^\perp$ and belong to $W^{1,1}(B(y_0,r_0)\cap X_h^\perp,\gamma_h^\perp)$, for any $y_0\in \Omega_h^\perp$.
\end{lemma}

\begin{proof}
Let us consider the function $g$; the argument
for $f$ is similar. 
We show that for any $y_0\in\Omega_h^\perp$ there exists $r_0>0$ such that $g\in{\rm Lip}(B(y_0,r_0))$. To this aim, let us fix $y_0\in\Omega_h^\perp$. Hence, there exists $t_0\in\R$ such that $x_0:=y_0+t_0h\in\Omega$, and we can find $r_0>0$ such that $B(x_0,2r_0)\subset \Omega$. We claim that $B(y_0,2r_0)\cap X_h^\perp\subset\Omega_h^\perp$ and $g(y)\leq t_0$ for any $y\in B(y_0,2r_0)\cap X_h^\perp$: indeed, $\|y+t_0h-x_0\|_X=\|y-y_0\|_X<2r_0$, and so $y+t_0h\in B(x_0,2r_0)\subset \Omega$. This implies that $y\in\Omega_h^\perp$ and $t_0\in\Omega_{y}$, which means $g(y)\leq t_0$ for any $y\in B(y_0,2r_0)\cap X_h^\perp$.
Hence, $g$ is convex and bounded from above on a symmetric domain. We claim that $g$ is bounded on $B(y_0,2r_0)\cap X_h^\perp$. Indeed, for any $y\in B(y_0,2r_0)\cap X_h^\perp$ let us consider $y'=y_0-(y-y_0)$. Then, we have
\begin{align*}
g(y_0)=g\left(\frac12y+\frac12y'\right)\leq \frac12g(y)+\frac12g(y')\leq \frac12 g(y)+\frac12 t_0.
\end{align*}
Hence, $g(y)\geq 2g(y_0)-t_0$. Since $B(y_0,r)+rB(0,1)= B(y_0,2r)$,
we infer that $g\in {\rm Lip}(B(y_0,r)\cap X_h^\perp)$ (see \cite[Proposition 1.6 and the successive Remark therein]{Phe93}).

The remain part follows from \cite[Theorems 5.11.1 and 5.11.2]{Bog98} and from the definition of Sobolev space $W^{1,1}(A,\gamma_h^\perp)$ with $A\subset X_h^\perp$ open set.
\end{proof}

{\begin{remark}\label{huxley}
We denote by $D_Gf$ and $D_Gg$ the G\^ateaux derivatives of $f$ and $g$, respectively, where they are defined, and analogously by $\nabla_Hf$ and $\nabla_Hg$ their $H$-derivatives where they are defined.
\begin{itemize}
\item [$(i)$]
The family $\mathscr A:=\{B(y_0,r_0)\cap X_h^\perp\subset \Omega_h^\perp: y_0\in\Omega_h^\perp, \ f,g\in{\rm Lip}_b(B(y_0,r_0)\cap X_h^\perp)\}$ is an open covering of $\Omega_h^\perp$. Since $X$ is separable, $\mathscr A$ admits a countable subcovering $\{B(y_n,r_n)\cap X_h^\perp\subset \Omega_h^\perp: y_n\in\Omega_h^\perp, \ f,g\in{\rm Lip}_b(B(y_n,r_n)\cap X_h^\perp), \ n\in\N\}$. 
Hence, $\nabla_Hf(y)$ and $\nabla_H g(y)$ (and also $D_Gf(y)$ and $D_Gg(y)$) are defined $\gamma_h^\perp$-a.e. $y\in \Omega_h^\perp$ and for such a values of $y$ we have $\nabla_Hf(y)=R_{\gamma}D_Gf(y)$ and $\nabla_Hg(y)=R_{\gamma}D_Gg(y)$.
\item [$(ii)$]
From \cite[Corollary 1.4]{AlbMaRoc97} there exists a partition of unity of Lipschitz functions subordinated to $\{B(y_n,r_n)\cap X_h^\perp:n\in\N\}$, i.e., there exists an open locally finite covering $\{A_n:n\in\N\}$ of $\Omega_h^\perp$ such that for any $n\in\N$ there exists $m=m(n)$ with $\overline{A_n}\subset B(y_m,r_m)\cap X_h^\perp$, and there exists a family $\{\psi_n:n\in\N\}\subset {\rm Lip}_b(X_h^\perp)$ such that ${\rm supp}(\psi_n)\subset A_n$ for any $n\in\N$, $\psi_n\geq 0$ for any $n\in\N$ and $\sum_{n\in\N}\psi_n=1$.
\end{itemize}
\end{remark}

Now we are ready to show the link between $D_\gamma \1_\Omega$ and $f$ and $g$.
\begin{lemma}
\label{quattro}
Let $\Omega$, $\Omega_h^\perp$, $f$ and $g$ as above. Then, 
\begin{align}
\label{aldous}
D_\gamma\1_\Omega=-\nu_{f}\mathscr S^{\infty-1}\res\Gamma(f,\Omega_h^\perp)+\nu_g\mathscr S^{\infty-1}\res\Gamma(g,\Omega_h^\perp),
\end{align}
where $\nu_f$ and $\nu_g$ have been defined in \eqref{ricola}. 
\end{lemma}

\begin{proof}
Let $\varphi\in \mathcal F C_b^1(X)$. Since $\Omega$ has finite perimeter, for any $k\in H$ we have
\begin{align*}
\int_\Omega\partial_k^*\varphi d\gamma=-\int_X\varphi d[D_\gamma\1_\Omega,k]_H.
\end{align*}
From Proposition \ref{business} with $D$ in place of $\Omega_h^\perp$, we know that 
both ${\rm Epi}(g,\Omega_h^\perp)$ and ${\rm Epi}(f,\Omega_h^\perp)$ have finite perimeter, 
and $\Omega={\rm Epi}(g,\Omega_h^\perp)\setminus(\Gamma(f,\Omega_h^\perp)\cup{\rm Epi}(f,\Omega_h^\perp))$. Therefore,
\begin{align*}
\int_{\Omega}\partial_k^*\varphi d\gamma
=&  \int_{{\rm Epi}(g,\Omega_h^\perp)}\partial_k^*\varphi d\gamma-\int_{{\rm Epi}(f,\Omega_h^\perp)}\partial_k^*\varphi d\gamma \\
= & -\int_X\varphi d[D_\gamma\1_{{\rm Epi}(g,\Omega_h^\perp)},k]_H
+ \int_X\varphi d[D_\gamma\1_{{\rm Epi}(f,\Omega_h^\perp)},k]_H \\
=  & -\int_X\varphi d[\nu,k]_H,
\end{align*}
since Lemma \ref{cementine} gives $\gamma(\Gamma(f,\Omega_h^\perp))=0$. Here, 
$\nu=D_\gamma\1_{{\rm Epi}(g,\Omega_h^\perp)}-D_\gamma\1_{{\rm Epi}(f,\Omega_h^\perp)}$. Therefore,
\begin{align}\label{hugo}
D_\gamma\1_\Omega=D_\gamma\1_{{\rm Epi}(g,\Omega_h^\perp)}-D_\gamma\1_{{\rm Epi}(f,\Omega_h^\perp)}.
\end{align}
{By the finiteness of the perimeter of ${\rm Epi}(g,\Omega_h^\perp)$ we have that $|D_\gamma\1_{{\rm Epi}(g,\Omega_h^\perp)}|= \mathscr S^{\infty-1}\res\Gamma(g,\Omega_h^\perp) $
is a finite measure. Further, for any $\varphi\in {\rm Lip}_b(X_h^\perp)$ such that $\overline{{\rm supp}(\varphi)}\subset B(y_{m(n)},r_{m(n)})\cap X_h^\perp$ for some $n\in\N$, any $\theta\in{\rm Lip}_b(X)$ and any $k\in H$ we have
\begin{align}
\int_X\theta(x)\varphi(x-\pi_hx) [D_\gamma\1_{{\rm Epi}(g,\Omega_h^\perp)},k]_H(dx)
= & - \int_X\1_{{\rm Epi}(g,\Omega_h^\perp)}\partial_k^*(\theta(x)\varphi(x-\pi_hx))\gamma(dx) \notag \\
= & -  \int_X\1_{{\rm Epi}(g,B(y_{m(n)},r_{m(n)})\cap X_h^\perp)}\partial_k^*(\theta(x)\varphi(x-\pi_hx))\gamma(dx) \notag \\
= & \int_X\theta(x)\varphi(x-\pi_hx)[D_\gamma\1_{{\rm Epi}(g,B(y_{m(n)},r_{m(n)})\cap X_h^\perp)},k]_H(dx).\label{brunone}
\end{align}
By density equality \eqref{brunone} holds for any $\theta \in\B_b(X)$. Let $\{\psi_n:n\in\N\}$ be the partition of unity introduced in Remark \ref{huxley} $(ii)$ and let $B\in \mathcal B(X)$. 
We have that $\psi_n\geq0$ everywhere for any $n\in\N$, so
\begin{align*}
 \sum_{n\in\N}\int_{X}\psi_n d\mathscr S^{\infty-1}\res\Gamma(g,\Omega_h^\perp)<\infty.
\end{align*}}

Since $g,f\in W^{1,1}({B(y_{m(n)},r_{m(n)})\cap X_h^\perp})$,
taking into account Theorem \ref{bicicletta} and \eqref{brunone} we have
\begin{align*}
 \int_B [\nu_g,k]_H&d\mathscr S^{\infty-1}\res\Gamma(g,\Omega_h^\perp) \\
= & \int_{X}\sum_{n\in\N}\psi_n(x-\pi_hx)\1_B(x) [\nu_g(x),k]_H\mathscr S^{\infty-1}\res\Gamma(g,\Omega_h^\perp)(dx)\\
= & \sum_{n\in\N}\int_{X}\psi_n(x-\pi_hx)\1_B(x) [\nu_g(x),k]_H\mathscr S^{\infty-1}\res\Gamma(g,\Omega_h^\perp)(dx)\\
= & \sum_{n\in\N}\int_{X}\psi_n(x-\pi_hx)\1_B(x) d[\nu_g(x),k]_H \mathscr S^{\infty-1}\res\Gamma(g,B(y_{m(n)},r_{m(n)})\cap X_h^\perp) (dx)\\
= & \sum_{n\in\N}\int_{X}\psi_n(x-\pi_hx)\1_B(x)  [D_\gamma\1_{{\rm Epi}(g,\B(y_{m(n)},r_{m(n)})\cap X_h^\perp)},k]_H (dx)\\
= & \sum_{n\in\N}\int_{X}\psi_n(x-\pi_hx)\1_B(x)  [D_\gamma\1_{{\rm Epi}(g,\Omega_h^\perp)},k]_H (dx)\\
= & \int_X\sum_{n\in\N}\psi_n(x-\pi_hx)\1_B (x) [D_\gamma\1_{{\rm Epi}(g,\Omega_h^\perp)} ,k]_H(dx)\\
= &\int_B d[D_\gamma\1_{{\rm Epi}(g,\Omega_h^\perp)},k]_H,
\end{align*}
where $\nu_g$ has been defined in \eqref{ricola} and we can change series and integral thanks to the dominated convergence theorem. This shows that
\begin{align*}
D_\gamma\1_{{\rm Epi}(g,\Omega_h^\perp)}=\nu_g\mathscr S^{\infty-1}\res\Gamma(g,\Omega_h^\perp)=\frac{-\nabla_H g(y)+h}{\sqrt{1+|\nabla_H g(y)|_H^2}}\mathscr S^{\infty-1}\res \Gamma(g,\Omega_h^\perp).
\end{align*}
The same argument applied to $f$ gives
\begin{align*}
D_\gamma\1_{{\rm Epi}(f,\Omega_h^\perp)}=\nu_f\mathscr S^{\infty-1}\res\Gamma(f,\Omega_h^\perp)=\frac{-\nabla_H f(y)+h}{\sqrt{1+|\nabla_H f(y)|_H^2}}\mathscr S^{\infty-1}\res \Gamma(f,\Omega_h^\perp),
\end{align*}
and the thesis follows from \eqref{hugo}.
\end{proof}
}
\begin{rmk}
We cannot directly apply \eqref{fedor} to \eqref{hugo} since $f$ and $g$ do not belong to $W^{1,1}(\Omega_h^\perp,\gamma_h^\perp)$, but they belong to $W^{1,1}(B(y_n,r_n)\cap X_h^\perp,\gamma_h^\perp)$ with $n\in\N$. Hence, we don't have global summability and we have to use the partition of unity.
\end{rmk}

Since $\Omega$ is an open convex set, $\p$ is defined everywhere and $\partial \Omega=\{x\in X:\p(x)=1\}$. Moreover, it follows that $\p$ is a continuous convex function.
Our aim is to prove that $\p(x)$ is G\^ateaux differentiable $\mathscr S^{\infty-1}$-a.e. $x\in\partial \Omega$.
We recall a characterization of G\^ateaux differentiability of a continuous convex function (see \cite[Proposition 1.8]{Phe93}).
\begin{proposition}\label{aperitivo}
Let $x_0\in X$. A continuous convex function $\psi$ defined on an open set $D\ni x_0$ is G\^ateaux differentiable at $x_0$ if and only if there exists a unique linear functional $x^*\in X^*$ such that
\begin{align*}
x^*(x-x_0)\leq \psi(x)-\psi(x_0), \quad \forall x\in D.
\end{align*} 
In this case, $x^*=d\psi(x_0)$.
\end{proposition}
In particular, by Lemma \ref{james}
for any $y\in B(\tilde y,r_{\tilde y})$, for $\mathscr S^{\infty-1}$-a.e. $\tilde y\in \Omega_h^\perp$ and suitable $r_{\tilde y}>0$ we have
\begin{align*}
- D_Gf(\tilde y)(y-\tilde y) \leq -f(y)+f(\tilde y), \quad D_Gg(\tilde y)(y-\tilde y) \leq g(y)-g(\tilde y), 
\end{align*}
where $D_Gf$ and $D_Gg$ is the G\^ateaux differential of $f$ and $g$, respectively.

We introduce the following notation. Let $y^*\in (X_h^\perp)^*$, let $h\in QX^*$ and let $h^*\in X^*$ such that $Qh^*=h$.
Then, $x^*:=y^*\otimes h^*\in X^*$ denotes the element of $X^*$ such that $x^*(x)=y^*(y)+t$ for any $x=y+th$, with $y\in X_h^\perp$ and $t\in\R$.

Now we have all the ingredients to prove the G\^ateaux differentiability of $\p$.

\begin{theorem}\label{grasso}
In our setting, let $x\in\Gamma(f,\Omega_h^\perp)$ such that $f$ is G\^ateaux differentiable at $y$, where $x=y+f(y)h$. Then, it holds that
\begin{align}
\label{cena}
D_G\p(x)=\frac{-D_Gf(y)\otimes h^*}{(-D_Gf(y)\otimes h^*)(x)}.
\end{align}

Analogously, if $x\in\Gamma(g,\Omega_h^\perp)$ and $g$ is G\^ateaux differentiable at $y$, where $x=y+g(y)h$, then we get
\begin{align}
\label{cenetta}
D_G\p(x)=\frac{D_Gg(y)\otimes -h^*}{(D_Gg(y)\otimes-h^*)(x)}.
\end{align}
In particular, $\p$ is G\^ateaux differentiable and $H$-differentiable for $\mathscr S^{\infty-1}$-a.e. $x\in \partial \Omega$, and
\begin{align}
\label{tazza}
\nabla_H\p(x)=
\begin{cases}
\displaystyle
\frac{-\nabla_Hf(y)\otimes h}{( -D_Gf(y)\otimes h^*)(x)},  & x=y+f(y)h, \ {\textrm{ $f$ G\^ateaux diff. at $y$}}, \vspace{2mm}\\
\displaystyle 
\frac{\nabla_Hg(y)\otimes - h}{( D_Gg(y)\otimes- h^*)(x)}, & x=y+g(y)h, \ {\textrm{ $g$ G\^ateaux diff. at $y$}}.
\end{cases}
\end{align}
\end{theorem}
\begin{proof}
We fix $x_0\in\Gamma(f,\Omega_h^\perp)$ such that $f$ is G\^ateaux differentiable at $y_0$, with $x_0:=y_0+f(y_0)h$ and $y_0\in\Omega_h^\perp$. Since $\p$ is continuous, from Proposition \ref{stereo} we know that $\partial \p(x_0)$ is nonempty. We claim that any element of $\partial \p(x_0)$ equals \eqref{cena}.
If the claim is true, by Proposition \ref{aperitivo} it follows that $\p$ is G\^ateaux differentiable at $x_0$ and $D_G\p(x_0)$ satisfies \eqref{cena}. Hence, it remains to prove the claim.

Let $x^*\in \partial \p(x_0)$. Lemma \ref{pasta} implies that $x^*\in C(\p)$, i.e., $x^*(x)\leq \p(x)$ for any $x\in X,$ and $x^*(x) =\p(x_0)=1$. Since $y_0\in\Omega_h^\perp$ and $\Omega^\perp_h$ is an open set, there exists $r>0$ such that,
for any $y\in B(y_0,r)\subset {\Omega_h^\perp}$, the element $x:=y+f(y)h\in\Gamma(f,\Omega_h^\perp)\subset\partial \Omega$. Therefore, $x^*(x)\leq \p(x)=1$ and
\begin{align*}
0
\geq &  x^*(x) -x^*(x_0)
=x^*( x-x_0) 
= x^*(y+f(y)h-y_0-f(y_0)h)
= x^*(y-y_0)+x^*(h) (f(y)-f(y_0),
\end{align*}
which implies that
\begin{align} \label{caramelle}
x^*(y-y_0) \leq x^*(h)(f(y_0)-f(y)).
\end{align}
Let us show that $x^*(h)>0$. 
Indeed, if by contradiction we assume that 
$x^*(h)\leq0$, then for any $t<0$ we have
\begin{align*}
\p(x_0+th)\geq x^*(x_0+th)=x^*(x_0)+t x^*(h)\geq 1.
\end{align*}
This means that $x_0+th=y_0+(t+f(y_0))h\notin \Omega $ for any $t<0$. This contradicts the fact that $y_0+ch\in \Omega$ for any $c\in (g(y_0),f(y_0))$, since $y_0\in\Omega_h^\perp$.
We have therefore proved that $x^*(h) >0$. Dividing both sides of \eqref{caramelle} by $x^*(h)$ we get 
\begin{align*}
z^*(y-y_0)\leq (-f)(y)-(-f)(y_0), \quad \forall y\in B(y_0,r),
\end{align*}
where $z^*:=( x^*(h))^{-1}x^*$. Since $(-f)$ is G\^ateaux differentiable at $y_0$, Proposition \ref{aperitivo} gives $z^*=D_G(-f)(y_0)=-D_Gf(y_0)$ on $X_h^\perp$. Now we compute $ x^*(h)$. From $x^*(x_0)=1$, we get
\begin{align*}
1=x^*(x_0)
=x^*(y_0)+f(y_0)x^*(h)=-D_Gf(y_0)(y_0)x^*(h)+f(y_0)x^*(h). 
\end{align*}
Hence,
\begin{align}
\label{alfred}
x^*(h)=\frac1{-D_Gf(y_0)(y_0)+f(y_0)}
 =\frac1{(-D_Gf(y_0)\otimes h^*)(x_0)}.
\end{align}
We are almost done. Indeed, for any $x\in X$, we consider the decomposition $x=y+th$ with $y\in X_h^\perp$ and $t\in \R$. Above computations reveal
\begin{align*}
x^*(x)= & x^*(y)+t x^*(h)
= \frac{ -D_Gf(y_0)(y)}{( -D_Gf(y_0)\otimes h^*)(x_0)}+\frac{t}{( -D_Gf(y_0)\otimes h^*)(x_0)} \\
= & \frac{ -D_Gf(y_0)(y) +t}{( -D_Gf(y_0)\otimes h^*)(x_0)} 
= \frac{ (-D_Gf(y_0)\otimes h^*)(x)}{( -D_Gf(y_0)\otimes h^*)(x_0)}.
\end{align*} 
The claim is therefore proved. The same arguments applied to $g$ give \eqref{cenetta}.

We have proved that $\gamma_h^\perp$a.e. $y\in\Omega_h^\perp$ the function $\p$ is G\^ateaux differentiable at $x_f=y+f(y)h$ and $x_g=y+g(y)h$. Equivalently, there exists a $\gamma_h^\perp$-negligible set $V\subset\Omega_h^\perp$ such that $\p$ is G\^ateaux differentiable on $\Gamma(f,\Omega_h^\perp\setminus V)\cup \Gamma(g,\Omega_h^\perp\setminus V)$. Moreover,
\begin{align*}
\int_{\Gamma(f,\Omega_h^\perp)}\1_{\Gamma(f,V)}d\mathscr S^{\infty-1}
= \int_{\Omega_h^\perp}\1_{\Gamma(f,V)}(y+f(y)h)G_1(f(y))\sqrt{1+|\nabla_Hf(y)|_H^2}\gamma_h^\perp(dy)=0,
\end{align*}
since $\1_{\Gamma(f,V)}(y+f(y)h)=\1_V(y)$ and $\gamma_h^\perp(V)=0$. This gives $\mathscr S^{\infty-1}(\Gamma(f,V))=0$ and, analogously, we get $\mathscr S^{\infty-1}(\Gamma(g,V))=0$. From Lemma \ref{cementine} we infer that $\p(x)$ is G\^ateaux differentiable $\mathscr S^{\infty-1}$-a.e. $x\in\partial \Omega$.
The last part of the statement follows because, as recalled in Remark \ref{huxley}$(i)$, $\nabla_H=R\gamma D_G$.
\end{proof}

Now we are ready to prove Theorem \ref{triciclo}.

\begin{proof}[of Theorem \ref{triciclo}]
By the last part of Theorem \ref{grasso}, $\nabla_H\p$ is defined and non-zero $\mathscr {S}^\infty-1$-almost everywhere on $\partial\Omega$. 

As a consequence of \eqref{alfred} we deduce that $(-D_Gf(y_0)\otimes h^*)(x_0)>0$ for any $y_0\in\Omega_{h}^\perp$ such that $x_0=y_0+f(y_0)h\in\Gamma(f,\Omega_h^\perp)$ and $f$ is differentiable at $y_0$, and $(D_Gg(y_0)\otimes (- h^*))(x_0)>0$ for any $y_0\in\Omega_{h}^\perp$ such that such that $x_0=y_0+g(y_0)h\in\Gamma(g,\Omega_h^\perp)$ and $g$ is differentiable at $y_0$.
Hence, \eqref{tazza} gives
\begin{align}\label{ritardo}
\frac{\nabla_H\p(x)}{|\nabla_H\p(x)|_H}=\nu_f(x),
\end{align}
if $x\in \Gamma(f,\Omega_h^\perp)$ and $f$ is differentiable at $y$, with $x=y+f(y)h$, and
\begin{align}\label{investimento}
\frac{\nabla_H\p(x)}{|\nabla_H\p(x)|_H}=-\nu_g(x),
\end{align}
if $x\in \Gamma(g,\Omega_h^\perp)$ and $g$ is differentiable at $y$, with $x=y+g(y)h$. 
Let $k\in H$ and let $\psi\in{\rm Lip}_b(X)$.
From \eqref{aldous}, \eqref{ritardo} and \eqref{investimento} we get
\begin{align*}
\int_\Omega\partial_k^*\psi d\gamma
= & -\int_X\psi d[ D_\gamma\1_\Omega,k]_H
= \int_{\Gamma(f,\Omega_h^\perp)}\psi [ \nu_f,k]_H d \mathscr S^{\infty-1}-\int_{\Gamma(g,\Omega_h^\perp)}\psi [  \nu_g,k]_H d \mathscr S^{\infty-1} \\
= & \int_{\Gamma(f,\Omega_h^\perp)}\psi\frac{\partial_k\p}{|\nabla_H\p|}d \mathscr S^{\infty-1}+\int_{\Gamma(g,\Omega_h^\perp)}\psi \frac{\partial_k\p}{|\nabla_H\p|}d \mathscr S^{\infty-1}
=  \int_{\partial\Omega}\psi\frac{\partial_k\p}{|\nabla_H\p|}d \mathscr S^{\infty-1}.
\end{align*}
\end{proof}

\bibliographystyle{plain}
\nocite{*} 
\bibliography{biblio_minkowski.bib}

\begin{thebibliography}{10}

\bibitem{AdMeMi18}
Davide Addona, Menegatti Giorgio, and Michele Miranda~Jr.
\newblock Functions of bounded variations on open domains in {W}iener spaces.
\newblock {\em in preparation}.

\bibitem{AM88}
H.~Airault and P.~Malliavin.
\newblock Int\'egration g\'eom\'etrique sur l'espace de {W}iener.
\newblock {\em Bull. Sci. Math. (2)}, 112(1):3--52, 1988.

\bibitem{AlbMaRoc97}
Sergio Albeverio, Zhi-Ming Ma, and Michael R\"ockner.
\newblock Partitions of unity in {S}obolev spaces over infinite-dimensional
  state spaces.
\newblock {\em J. Funct. Anal.}, 143(1):247--268, 1997.

\bibitem{AF11}
Luigi Ambrosio and Alessio Figalli.
\newblock Surface measures and convergence of the {O}rnstein-{U}hlenbeck
  semigroup in {W}iener spaces.
\newblock {\em Ann. Fac. Sci. Toulouse Math. (6)}, 20(2):407--438, 2011.

\bibitem{AmbFusPal00}
Luigi Ambrosio, Nicola Fusco, and Diego Pallara.
\newblock {\em Functions of bounded variation and free discontinuity problems}.
\newblock Oxford Mathematical Monographs. The Clarendon Press, Oxford
  University Press, New York, 2000.

\bibitem{AmbMirManPal10}
Luigi Ambrosio, Michele Miranda, Jr., Stefania Maniglia, and Diego Pallara.
\newblock B{V} functions in abstract {W}iener spaces.
\newblock {\em J. Funct. Anal.}, 258(3):785--813, 2010.

\bibitem{AMP10}
Luigi Ambrosio, Michele Miranda, Jr., and Diego Pallara.
\newblock Sets with finite perimeter in {W}iener spaces, perimeter measure and
  boundary rectifiability.
\newblock {\em Discrete Contin. Dyn. Syst.}, 28(2):591--606, 2010.

\bibitem{AMP15}
Luigi Ambrosio, Michele Miranda, Jr., and Diego Pallara.
\newblock Some fine properties of {$BV$} functions on {W}iener spaces.
\newblock {\em Anal. Geom. Metr. Spaces}, 3:212--230, 2015.

\bibitem{Bog98}
Vladimir~I. Bogachev.
\newblock {\em Gaussian measures}, volume~62 of {\em Mathematical Surveys and
  Monographs}.
\newblock American Mathematical Society, Providence, RI, 1998.

\bibitem{Cappa}
Gianluca Cappa.
\newblock On the ornstein-uhlenbeck operator in convex sets of banach spaces.
\newblock {\em preprint, arXiv:1503.02836}.

\bibitem{CLMN12}
Vicent Caselles, Alessandra Lunardi, Michele Miranda~Jr, and Matteo Novaga.
\newblock Perimeter of sublevel sets in infinite dimensional spaces.
\newblock {\em Adv. Calc. Var.}, 5(1):59--76, 2012.

\bibitem{CL14}
Pietro Celada and Alessandra Lunardi.
\newblock Traces of {S}obolev functions on regular surfaces in infinite
  dimensions.
\newblock {\em J. Funct. Anal.}, 266(4):1948--1987, 2014.

\bibitem{DG53}
Ennio De~Giorgi.
\newblock Definizione ed espressione analitica del perimetro di un insieme.
\newblock {\em Atti Accad. Naz. Lincei. Rend. Cl. Sci. Fis. Mat. Nat. (8)},
  14:390--393, 1953.

\bibitem{FD92}
Denis Feyel and Arnaud de~La~Pradelle.
\newblock Hausdorff measures on the {W}iener space.
\newblock {\em Potential Anal.}, 1(2):177--189, 1992.

\bibitem{F00}
Masatoshi Fukushima.
\newblock {$BV$} functions and distorted {O}rnstein {U}hlenbeck processes over
  the abstract {W}iener space.
\newblock {\em J. Funct. Anal.}, 174(1):227--249, 2000.

\bibitem{FH01}
Masatoshi Fukushima and Masanori Hino.
\newblock On the space of {BV} functions and a related stochastic calculus in
  infinite dimensions.
\newblock {\em J. Funct. Anal.}, 183(1):245--268, 2001.

\bibitem{HI10}
Masanori Hino.
\newblock Sets of finite perimeter and the {H}ausdorff-{G}auss measure on the
  {W}iener space.
\newblock {\em J. Funct. Anal.}, 258(5):1656--1681, 2010.

\bibitem{LMP15}
Alessandra Lunardi, Michele Miranda, Jr., and Diego Pallara.
\newblock {$BV$} functions on convex domains in {W}iener spaces.
\newblock {\em Potential Anal.}, 43(1):23--48, 2015.

\bibitem{Phe93}
R.~R. Phelps.
\newblock {\em Convex functions, monotone operators and differentiability},
  volume 1364 of {\em Lecture Notes in Mathematics}.
\newblock Springer-Verlag, Berlin, second edition, 1993.

\bibitem{P81}
D.~Preiss.
\newblock Gaussian measures and the density theorem.
\newblock {\em Comment. Math. Univ. Carolin.}, 22(1):181--193, 1981.

\end{thebibliography}
\markboth{\textsc{References}}{\textsc{References}}
\end{document}